\documentclass{amsart}

\usepackage{amssymb,amsmath,amsthm,latexsym,booktabs,pxfonts,youngtab,hyperref,bbm}

\theoremstyle{definition}
\newtheorem{definition}{Definition}[section]
\newtheorem{remark}[definition]{Remark}
\newtheorem{openproblem}[definition]{Open Problem}

\theoremstyle{plain}
\newtheorem{lemma}[definition]{Lemma}
\newtheorem{proposition}[definition]{Proposition}
\newtheorem{theorem}[definition]{Theorem}
\newtheorem{conjecture}[definition]{Conjecture}

\newcommand{\field}{\mathbb{Q}}

\allowdisplaybreaks

\begin{document}

\title{Jordan quadruple systems}

\author{Murray Bremner}

\address{Department of Mathematics and Statistics, University of Saskatchewan, Canada}

\email{bremner@math.usask.ca}

\author{Sara Madariaga}

\address{Department of Mathematics and Statistics, University of Saskatchewan, Canada}

\email{madariaga@math.usask.ca}

\dedicatory{To our friend and colleague Luiz Antonio Peresi on his retirement}

\thanks{Murray Bremner was supported by a Discovery Grant from NSERC, the Natural Sciences and Engineering Research Council.
Sara Madariaga was supported by a Postdoctoral Fellowship from PIMS, the Pacific Institute for Mathematics Sciences.}



\begin{abstract}
We define Jordan quadruple systems by the polynomial identities of degrees 4 and 7 satisfied by
the Jordan tetrad $\{ a,b,c,d \} = abcd + dcba$ as a quadrilinear operation on associative algebras.
We find further identities in degree 10 which are not consequences of the
defining identities.
We introduce four infinite families of finite dimensional Jordan quadruple systems, and construct
the universal associative envelope for a small system in each family.
We obtain analogous results for the anti-tetrad $[a,b,c,d] = abcd - dcba$.
Our methods rely on computer algebra, especially linear algebra on large matrices,
the LLL algorithm for lattice basis reduction, representation theory of the symmetric group,
noncommutative Gr\"obner bases, and Wedderburn decompositions of associative algebras.
\end{abstract}

\maketitle

\allowdisplaybreaks


\section{Introduction}

In this paper we study the quadrilinear operations $\{a,b,c,d\} = abcd + dcba$ and $[a,b,c,d] = abcd - dcba$
in associative algebras.
The first is the Jordan tetrad which plays an important role in the structure theory of
Jordan algebras \cite{MZ,Z}.
The second is the anti-tetrad, which seems not to have been studied until now.

\subsection{Motivation}

In an associative algebra for $n \ge 2$ we define the $n$-tad to be this $n$-ary multilinear operation:
$\{ a_1, \dots, a_n \} = a_1 \cdots a_n \, + \, a_n \cdots a_1$.
For $n = 2$ we obtain the Jordan product $\{ a, b \} = ab + ba$ satisfying commutativity and the Jordan identity:
  \[
  \{ a, b \} \equiv \{ b, a \},
  \qquad
  \{ \{ \{ a, a \}, b \}, a \} \equiv \{ \{ a, a \}, \{ b, a \} \}.
  \]
There are further ``special'' identities satisfied by the Jordan product in every associative algebra
which do not follow from the defining identities; the simplest occur in degrees 8 and 9 and are
called the Glennie identities \cite{Glennie1,Glennie2}.
A Jordan algebra is ``special'' if it can be represented as a subspace of an associative algebra closed
under the Jordan product; otherwise, it is ``exceptional''.
If a special Jordan algebra is finite dimensional then its universal associative enveloping algebra is also
finite dimensional.
A survey of the role of identities in Jordan theory has been given by McCrimmon \cite{M2}.
For the structure and representation theory of finite dimensional Jordan algebras, see Jacobson \cite{Jacobson}.
For the modern theory including infinite dimensional algebras, see McCrimmon \cite{Mbook}.

For $n = 3$ we obtain the Jordan triple product $abc + cba$; in every associative algebra, this operation satisfies identities
which define Jordan triple systems (JTS):
  \[
  \{ a, b, c \} \equiv \{ c, b, a \},
  \qquad
  \{ \{ a, b, c \}, d, e \} \equiv
  \{ \{ a, d, e \}, b, c \}
  - \{ a, \{ b, e, d \}, c \}
  + \{ a, b, \{ c, d, e \} \}.
  \]
In contrast to the Jordan identity these identities are multilinear.
There are special identities in higher degree: identities satisfied by the Jordan triple product in every associative algebra
but which do not follow from the defining identities \cite{LM,M1}.
For the classification of finite dimensional JTS, see \cite{Loos,Meyberg,Neher}
and for their universal associative envelopes, see \cite{Loos2}.

Closely related to Jordan triple systems are the anti-Jordan triple systems (AJTS), see \cite{FF}.
Finite dimensional simple AJTS have been classified \cite{Bashir}.
These systems are defined by identities satisfied by the anti-Jordan triple product
$abc - cba$ in every associative algebra:
  \[
  \{ a, b, c \} + \{ c, b, a \} \equiv 0,
  \qquad
  \{ \{ a, b, c \}, d, e \} \equiv
  \{ \{ a, d, e \}, b, c \}
  + \{ a, \{ b, e, d \}, c \}
  + \{ a, b, \{ c, d, e \} \}.
  \]
Universal associative envelopes for one infinite family of simple AJTS have been constructed \cite{Elgendy2}.

At the next step $n = 4$ we obtain  the Jordan tetrad $abcd + dcba$, which arises in the study of symmetric elements
of associative algebras.
Let $A_n$ be the free unital associative algebra on generators $x_1, \dots, x_n$ over a field $F$.
The involution $\ast\colon A_n \to A_n$ defined on monomials by
$( x_{i_1} \cdots x_{i_d} )^\ast = x_{i_d} \cdots x_{i_1}$
satisfies $(ab)^\ast = b^\ast a^\ast$ and $(a^\ast)^\ast = a$.
The subspace $H_n = \{ \, a \in A_n \mid a^\ast = a \, \}$ of symmetric elements is a Jordan algebra under the
Jordan product.
Let $SJ_n$ be the free special Jordan algebra: the Jordan subalgebra of $H_n$ generated by
$x_1, \dots, x_n$.
Cohn \cite{Cohn1,Cohn2} has shown that $SJ_n = H_n$ for $n \le 3$;
but for $n \ge 4$, $SJ_n \ne H_n$ and $H_n$ is generated as a Jordan algebra by $x_1, \dots, x_n$ and the tetrads
$x_{i_1} x_{i_2} x_{i_3} x_{i_4} + x_{i_4} x_{i_3} x_{i_2} x_{i_1}$ for $1 \le i_1 < i_2 < i_3 < i_4 \le n$.
Since $A_n$ is unital, $H_n$ is generated by $x_1, \dots, x_n$ using the tetrad as a quadrilinear operation.
From the tetrad we recover the Jordan product by setting two arguments to 1, and so
this operation provides the natural algebraic structure on $H_n$ for $n \ge 4$.

\begin{definition}
The \textbf{tetrad} and \textbf{anti-tetrad} are respectively these quadrilinear operations on associative algebras:
  \[
  \{a,b,c,d\} = abcd + dcba,
  \qquad
  [a,b,c,d] = abcd - dcba.
  \]
\end{definition}

\subsection{Outline}

In \S\ref{sectioncomputational} we recall basic results from the representation theory of the symmetric group,
emphasizing a computational point of view, with a focus on applications to polynomial identities.
In \S\ref{sectiondegree7} we use computer algebra to determine a complete set of generators for the
multilinear polynomial identities of degrees 4 and 7 satisfied by the tetrad in every associative algebra.
These identities define Jordan quadruple systems (JQS).
In \S\ref{sectiondegree10} we use representation theory to show that there are further ``special''
identities in degree 10 satisfied by the tetrad in every associative algebra which do not follow from the defining identities for JQS.
We use the LLL algorithm for lattice basis reduction to obtain five explicit nonlinear special identities.
In \S\ref{sectionUtetrad} we introduce four infinite families of finite dimensional JQS.
For one system in each family, we use noncommutative Gr\"obner bases to construct its universal
associative envelope; in each case the envelope is finite dimensional, and we use the Wedderburn decomposition
of an associative algebra to classify the finite dimensional irreducible representations.
In \S\S\ref{sectionAJQS}--\ref{sectionUantitetrad} we describe analogous results for the anti-tetrad.
Throughout the paper we suggest a number of open problems as possible directions for further research.

\subsection{Conventions}

Unless otherwise indicated, all computations are performed with the computer algebra system \texttt{Maple}
using arithmetic over either the ring $\mathbb{Z}$ of integers or the field $\mathbb{Q}$ of rational numbers.
To save computer memory, we often use arithmetic over the finite field $\mathbb{F}_p$ for some prime $p$,
followed by rational reconstruction to recover results over $\mathbb{Z}$.
We use the symbol $\equiv$ to indicate that an equation holds for all values of the arguments.


\section{Computational methods for polynomial identities} \label{sectioncomputational}

In this section we review rather informally the structure theory of the group algebra $\field S_n$ from an algorithmic
point of view, with a focus on applications to polynomial identities.
For a more detailed exposition, see \cite[\S 5]{BP2011}.

\subsection{Representation matrices} \label{subsectionrepmat}

If $\mathbb{F}$ is a field of characteristic 0 or $p > n$, then the group algebra $\mathbb{F} S_n$ is semisimple, and
decomposes as the direct sum of simple two-sided ideals, each isomorphic to a full matrix algebra:
  \[
  R\colon \mathbb{F} S_n \xrightarrow{\quad\approx\quad} \bigoplus_\lambda M_{d_\lambda }(\mathbb{F})
  \]
The sum is over all partitions $\lambda$ of $n$, and we write $R_\lambda\colon \mathbb{F} S_n \to M_{d_\lambda }(\mathbb{F})$
for the projection onto component $\lambda$.
The dimension $d_\lambda $ of the irreducible representation $[\lambda]$ can be computed from the
Young diagram of $\lambda$ using the hook formula.
Given any permutation $\sigma \in S_n$, the $d_\lambda \times d_\lambda$ representation matrix $R_\lambda(\sigma)$
in the natural representation has entries in $\{0,1,-1\}$ and can be efficiently computed using the methods
of \cite{Clifton} and \cite[Figure 1]{BP2011}.

We study multilinear polynomial identities for nonassociative quadrilinear operations.
Monomials in such an operation have degrees $n \equiv 1$ (mod 3).
By an association type in degree $n$ we mean a placement of operation symbols in a sequence of $n$ arguments
(without specifying the arguments).
By the identity monomial for an association type in degree $n$ we mean the monomial with the identity permutation
$a_1 \cdots a_n$ of the arguments.
In general, a monomial in degree $n$ consists of an association type applied to a permutation of the arguments.
If there are $t = t(n)$ distinct association types in degree $n$, totally ordered in some way,
then any multilinear polynomial $I$
of degree $n$ can be written as a sum of $t$ components $I_1 + \cdots + I_t$; in each component, the terms differ only by
the permutation of the arguments.
We can therefore regard $I$ as an element of the direct sum of $t$ copies of the group algebra, $( \mathbb{F} S_n )^t$,
on which $S_n$ acts by left multiplication.
For each partition $\lambda$, we apply the projection $R_\lambda$ to each component $I_1, \dots, I_t$ to obtain
a sequence of $t$ matrices of size $d_\lambda \times d_\lambda$ which we combine horizontally into a matrix of size
$d_\lambda \times t d_\lambda$.
This is the representation matrix for the component of $I$ in partition $\lambda$.
The row canonical form (RCF) of this matrix is the normal form of $I$ for partition $\lambda$.
(We assume that zero rows have been removed from the RCF, so that a matrix in RCF always has full rank.)
Each row of the RCF generates a submodule of $( \mathbb{F} S_n )^t$ isomorphic to $[\lambda]$.
Hence the rank of this matrix is the multiplicity of $[\lambda]$ in the submodule of $( \mathbb{F} S_n )^t$ generated by $I$.

More generally, we consider a sequence $I^{(1)}, \dots, I^{(s)}$ of $s$ multilinear identities in degree $n$ and
the $s d_\lambda \times t d_\lambda$ matrix in which the $(i,j)$ block is the image under $R_\lambda$ of the terms of the
$i$-th identity in the $j$-th association type.
The rows of the RCF of this matrix provide a canonical set of generators for the isotypic component of type $[\lambda]$
generated by $I^{(1)}, \dots, I^{(s)}$ in $( \mathbb{F} S_n )^t$.

\subsection{Polynomial identities}

The process of finding a complete set of generators for the $S_n$-module of multilinear identities in degree $n$ satisfied by
the tetrad $\{a_1,a_2,a_3,a_4\}$ consists of three steps.
(For the anti-tetrad, the process is similar, but we must also keep track of sign changes resulting from reversal.)
These computations take place in the multilinear subspaces $\mathrm{Quad}(n)$ of degree $n$ in
the free quaternary algebra with one operation satisfying
$\{a_1,a_2,a_3,a_4\} - \{a_4,a_3,a_2,a_1\} \equiv 0$.
Since $\mathrm{Quad}(n)$ is also a left $S_n$-module (permutations act on the subscripts),
we can regard it as the degree $n$ component of the symmetric operad $\mathrm{Quad}$
generated by one quaternary operation for which the action of $S_4$ is given by the symmetry in degree 4;
that is, the generating space $\mathrm{Quad}(4)$ is the quotient of the regular $S_4$-module $\field S_4$
by the left ideal generated by $a_1a_2a_3a_4 - a_4a_3a_2a_1$.

\subsubsection*{Step 1}

We distinguish two types of identities satisfied by the tetrad in degree $n$ which are consequences of known identities
in lower degrees.

\smallskip
\noindent
$\bullet$ Type 1: Symmetries.
Since the tetrad satisfies symmetry in degree 4,
we use this to reduce the number of association types in degrees $n > 4$.
For example, in degree 7 we reduce the number of association types from 4 to 2 as follows:
  \begin{align*}
  &
  \{ a_1, a_2, \{ a_3, a_4, a_5, a_6 \}, a_7 \}
  \longrightarrow
  \{ a_7, \{ a_3, a_4, a_5, a_6 \}, a_2, a_1 \},
  \\
  &
  \{ a_1, a_2, a_3, \{ a_4, a_5, a_6, a_7 \} \}
  \longrightarrow
  \{ \{ a_4, a_5, a_6, a_7 \}, a_3, a_2, a_1 \}.
  \end{align*}
This does not eliminate all the identities in degree $n$ which follow from symmetry in degree 4,
since for each association type there remain identities of the form $\iota - \tau \equiv 0$
where $\iota$ is the identity monomial and $\tau$ is a monomial in the same association type
obtained from $\iota$ by a single application of symmetry.
For example, for one of the association types in degree 10 we have these identities:
  \begin{align*}
  &
  \{ \{ a_1, a_2, a_3, a_4 \}, a_5, a_6, \{ a_7, a_8, a_9, a_{10} \} \}
  -
  \{ \{ a_4, a_3, a_2, a_1 \}, a_5, a_6, \{ a_7, a_8, a_9, a_{10} \} \}
  \equiv 0,
  \\
  &
  \{ \{ a_1, a_2, a_3, a_4 \}, a_5, a_6, \{ a_7, a_8, a_9, a_{10} \} \}
  -
  \{ \{ a_1, a_2, a_3, a_4 \}, a_5, a_6, \{ a_{10}, a_9, a_8, a_7 \} \}
  \equiv 0,
  \\
  &
  \{ \{ a_1, a_2, a_3, a_4 \}, a_5, a_6, \{ a_7, a_8, a_9, a_{10} \} \}
  -
  \{ \{ a_7, a_8, a_9, a_{10} \}, a_6, a_5, \{ a_1, a_2, a_3, a_4 \} \}
  \equiv 0.
  \end{align*}
Symmetries of the association types generate a submodule $\mathrm{Symm}(n) \subset \mathrm{Quad}(n)$.

\smallskip
\noindent
$\bullet$ Type 2: Liftings.
We assume by induction that we have already determined a set of $S_{n-3}$-module generators for the multilinear identities
of degree $n{-}3$ satisfied by the tetrad.
(We do not include the symmetry in degree 4 in this process, so $n-3 > 4$.)
Let $I( a_1, \dots, a_{n-3} ) \equiv 0$ be such an identity.
We consider $n{-}3$ substitutions of a tetrad for an argument of $I$, and two embeddings of $I$ into a tetrad:
  \begin{align*}
  &
  I( \{ a_1, a_{n-2}, a_{n-1}, a_n \}, a_2, \dots, a_{n-3} ),
  \quad \dots, \quad
  I( a_1, \dots, \{ a_i, a_{n-2}, a_{n-1}, a_n \}, \dots, a_{n-3} ),
  \quad \dots,
  \\
  &
  I( a_1, a_2, \dots, \{ a_{n-3}, a_{n-2}, a_{n-1}, a_n \} ),
  \\
  &
  \{ I( a_1, a_2, \dots, a_{n-3} ), a_{n-2}, a_{n-1}, a_n \},
  \quad
  \{ a_{n-2}, I( a_1, a_2, \dots, a_{n-3} ), a_{n-1}, a_n \}.
  \end{align*}
These multilinear polynomials clearly vanish, and form a set of $S_n$-module generators for the
multilinear identities in degree $n$ which are consequences of $I$.
We repeat this process for every generator in degree $n{-}3$.
Iteration of this process produces a set of generators for the $S_n$-module of identities for the tetrad in degree $n$
which are consequences of the known identities of lower degree.
The liftings of the identities of lower degree generate a submodule $\mathrm{Lift}(n) \subset \mathrm{Quad}(n)$.

\begin{definition}
The submodule $\mathrm{Old}(n) = \mathrm{Symm}(n) + \mathrm{Lift}(n) \subset \mathrm{Quad}(n)$
consists of identities in degree $n$ which are consequences of known identities of lower degree.
\end{definition}

We use the methods of \S\ref{subsectionrepmat} to find canonical generators for $\mathrm{Old}(n)$.
For each $\lambda$, we obtain the matrix $\mathrm{RCF}_\lambda(\mathrm{Old}(n))$
whose rows are a set of independent generators for the simple summands $[\lambda]$
in the isotypic component of type $\lambda$ in $\mathrm{Old}(n)$.

\subsubsection*{Step 2}

We determine all the multilinear identities satisfied by the tetrad in degree $n$.
By \S\ref{subsectionrepmat}, these identities are elements of $\mathrm{Quad}(n) = ( \mathbb{F} S_n )^t$,
where $t = t(n)$ is the number of association types in degree $n$.
Each monomial in degree $n$ can be expanded by applying the definition $\{a,b,c,d\} = abcd + dcba$ to each occurrence of
the operation symbol; the result is a multilinear associative polynomial of degree $n$:
an element of $\mathbb{F} S_n$.
This process extends linearly to $\mathrm{Quad}(n)$, giving the expansion map
$E_n\colon \mathrm{Quad}(n) \to \mathbb{F} S_n$, which is an $S_n$-module morphism.
The kernel of $E_n$ is the submodule of $\mathrm{Quad}(n)$ consisting of those multilinear polynomials
which vanish after expansion into the free associative algebra.

\begin{definition}
The submodule $\mathrm{All}(n) \subset \mathrm{Quad}(n)$ is the kernel of $E_n\colon \mathrm{Quad}(n) \to \mathbb{F} S_n$:
all the multilinear identities in degree $n$ satisfied by the tetrad.
\end{definition}

We use the methods of \S\ref{subsectionrepmat} to find canonical generators for $\mathrm{All}(n)$.
For each $\lambda$, we calculate the matrix $\mathrm{RCF}_\lambda(\mathrm{All}(n))$
whose rows are a set of independent generators for distinct simple summands of type $[\lambda]$
in $\mathrm{All}(n)$.
We restrict $E_n$ to the corresponding isotypic component of its domain and codomain,
obtaining the map $E_n^\lambda\colon M_{d_\lambda}(\mathbb{F})^t \to M_{d_\lambda}(\mathbb{F})$.
The matrix representing $E_n^\lambda$ has size $d_\lambda \times t d_\lambda$, and the $i$-th block consists of
the representation matrix of the expansion of the $i$-th identity monomial.
We compute a canonical basis for the nullspace of $E_n^\lambda$ as follows:
\begin{itemize}
\item
we calculate $\mathrm{RCF}(E_n^\lambda)$, its rank $q$, and its nullity $t d_\lambda - q$;
\item
we set the free variables (columns without leading 1s) to the standard basis vectors in $\mathbb{F}^{t d_\lambda - q}$
and solve for the leading variables;
\item
we put the resulting basis vectors into a matrix of size $( t d_\lambda - q ) \times t d_\lambda$;
\item
we compute the RCF of this matrix, which is $\mathrm{RCF}_\lambda(\mathrm{All}(n))$.
\end{itemize}

\subsubsection*{Step 3}

We compare the results of Steps 1 and 2 to determine whether there exist new multilinear identities
satisfied by the tetrad in degree $n$.

\begin{definition}
The quotient module $\mathrm{New}(n) = \mathrm{All}(n) / \mathrm{Old}(n)$ consists of
the new identities for the tetrad in degree $n$:
a complete set of representatives for the equivalence classes of all identities modulo old identities.
\end{definition}

To obtain a canonical set of generators for $\mathrm{New}(n)$, for each $\lambda$
we compare the matrices $\mathrm{RCF}_\lambda(\mathrm{Old}(n))$
and $\mathrm{RCF}_\lambda(\mathrm{All}(n))$.
If these two matrices have the same rank, they must be equal;
this indicates that every identity in the isotypic component $\lambda$ follows from identities of lower degree.
If the ranks are not equal, then $\texttt{oldrank}(\lambda)$, the rank of $\mathrm{RCF}_\lambda(\mathrm{Old}(n))$,
must be strictly less than $\texttt{allrank}(\lambda)$, the rank of $\mathrm{RCF}_\lambda(\mathrm{All}(n))$;
moreover,
the row space of $\mathrm{RCF}_\lambda(\mathrm{Old}(n))$ must be a subspace of the row space of $\mathrm{RCF}_\lambda(\mathrm{All}(n))$.
The difference $\texttt{allrank}(\lambda) - \texttt{oldrank}(\lambda)$ is the multiplicity of $[\lambda]$ in
$\mathrm{New}(n)$.
To find $S_n$-module generators for the isotypic component $\lambda$ in $\mathrm{New}(n)$,
we identify the positions $(i_1,j_1),\dots,(i_\texttt{oldrank},j_\texttt{oldrank})$ of the leading 1s in
$\mathrm{RCF}_\lambda(\mathrm{Old}(n))$,
and the analogous positions $(i'_1,j'_1),\dots,(i'_\texttt{allrank},j'_\texttt{allrank})$
in $\mathrm{RCF}_\lambda(\mathrm{All}(n))$.
Comparing these two sets of positions, we obtain
  \[
  \mathcal{J} = \big\{ j_1, \dots, j_\texttt{oldrank} \big\} \subsetneq \big\{ j'_1, \dots, j'_\texttt{allrank} \big\} = \mathcal{J}',
  \quad
  \mathcal{J}' \setminus \mathcal{J} = \big\{ j'_{k_1}, \dots, j'_{k_{\texttt{allrank} - \texttt{oldrank}}} \big\}.
  \]
Rows $i'_{k_1}, \dots, i'_{k_{\texttt{allrank} - \texttt{oldrank}}}$ of $\mathrm{RCF}_\lambda(\mathrm{All}(n))$
are the canonical generators of $\mathrm{New}(n)$.

\subsection{Linear algebra over $\mathbb{Q}$, $\mathbb{F}_p$ and $\mathbb{Z}$}

In general, we prefer to do all computations over the field $\mathbb{Q}$ of rational numbers.
However, it is well known that, even if the original matrix is very sparse and its entries belong to $\{0,\pm 1\}$,
standard algorithms for computing the RCF produce exponential increases in the matrix entries.
Even if enough computer memory is available to store the intermediate results, the calculations can take far too
much time.
It is therefore often convenient to use modular arithmetic, so that each matrix entry uses a fixed small
amount of memory.
This leads to the problem of rational reconstruction: recovering the correct results over $\mathbb{Q}$ or $\mathbb{Z}$
from the known results over $\mathbb{F}_p$.

In general, rational reconstruction is extremely ill-defined: we want to compute an inverse for a
partially-defined infinity-to-one map.
Therefore it is only effective when we have a good theoretical understanding of the arithmetical nature of the
expected results.
In our computations, we may assume that the correct rational coefficients have a common highly composite denominator.
The reason is that the vector spaces we study are all modules over the symmetric group $S_n$, and in the Wedderburn
decomposition of the group algebra $\mathbb{Q} S_n$ into a direct sum of full matrix algebras, the matrix units are
linear combinations of permutations in which the coefficients have $n!$ as their common denominator.
Moreover, $\mathbb{F} S_n$ is semisimple whenever $\mathbb{F}$ has characteristic 0 or $p > n$, so formulas giving
the Wedderburn decomposition in rational arithmetic also apply to modular arithmetic with a prime larger than the
degree of the multilinear polynomial identities.
For a more precise statement of this fact, see \cite[Lemma 8]{BP2009c}.

If we use a large enough prime, say $p > n!$ rather than $p > n$, then we can recognize the common denominator
$d$ of the rational coefficients from the distribution of the congruence classes
modulo $p$: the modular coefficients are clustered near the congruence classes representing $i/d$ for $1 \le i \le d-1$.
This allows us to recover the rational coefficients; we then multiply by the LCM of the denominators to get
integer coefficients, and finally divide by the GCD of the coefficients to get the integer vector with least
Euclidean length which is a scalar multiple of the original coefficient vector.
Once we have a realistic conjecture for the correct integer coefficients, we can perform a much simpler computation
using rational arithmetic to verify the results.

Most of our computations involve finding a basis of integer vectors for the nullspace of a matrix with integer entries.
In some cases, the modular methods described in the previous paragraphs give good results, meaning that
the basis vectors have small Euclidean lengths.
In other cases, we obtain much better results using the Hermite normal form (HNF) of an integer matrix
(the analogue over $\mathbb{Z}$ of the RCF) together with the LLL algorithm for lattice basis reduction.
If $A$ is an $s \times t$ matrix over $\mathbb{Z}$ then computing the HNF of the transpose produces two
matrices over $\mathbb{Z}$: a $t \times s$ matrix $H$ and a $t \times t$ matrix $U$ with $\det(U) = \pm 1$
such that $U A^t = H$.
If $\mathrm{rank}(A) = r$ then the bottom $t-r$ rows of $U$ form a lattice basis for the left integer nullspace of $A^t$,
which is the right integer nullspace of $A$.
We then apply the LLL algorithm with increasing values of the parameter to this basis in order to obtain
shorter basis vectors.
We define a precise measure of the size of a lattice basis consisting of integer vectors $v_1, \dots, v_k$ by
the number of decimal digits in the product of the Euclidean lengths of the vectors:
  \begin{equation}
  \label{basissize}
  \sum_{i=1}^k \log_{10} \| v_i \|.
  \end{equation}
For a more detailed discussion of the application of HNF and LLL to polynomial identities, including algorithms
in pseudocode, see \cite[\S3]{BP2009} and the monograph \cite{BremnerLLL}.


\section{Defining identities for Jordan quadruple systems} \label{sectiondegree7}

In this section we determine a complete set of $S_n$-module generators for the multilinear polynomial
identities satisfied by the tetrad in degrees 4 and 7.

\begin{lemma} \label{tetrad4}
Every multilinear identity in degree 4 satisfied by the tetrad in every associative algebra is a consequence of
the symmetry $\{a,b,c,d\} - \{d,c,b,a\} \equiv 0$.
\end{lemma}

\begin{proof}
Consider two copies of the group algebra $\mathbb{Q} S_4$;
the first, denoted $Q$, is the multilinear subspace of degree 4 in the free quaternary algebra
with one operation $\{-,-,-,-\}$;
the second, denoted $A$, is the multilinear subspace of degree 4 in the free associative algebra.
Bases of $Q$ and $A$ are the sets
$\{ \, \{ a^\sigma, b^\sigma, c^\sigma, d^\sigma \} \mid \sigma \in S_4 \}$
and
$\{ \, a^\sigma b^\sigma c^\sigma d^\sigma \mid \sigma \in S_4 \}$,
ordered lexicographically.
We initialize the $24 \times 24$ expansion matrix $E$ in which the $(i,j)$ entry is the coefficient of the $i$-th
associative monomial in the expansion of the $j$-th quaternary monomial, and compute its RCF, which has rank 12,
and hence nullity 12.
We obtain a basis for the nullspace by setting the free variables equal to the standard basis vectors in $\mathbb{Q}^{12}$
and solving for the leading variables.
We put these basis vectors into a $12 \times 24$ matrix and compute its RCF;
every row represents a permutation of the stated symmetry.
\end{proof}

\begin{theorem} \label{tetrad7}
Every multilinear identity in degree 7 satisfied by the tetrad in every associative algebra is a consequence of
the symmetry and these three identities in degree 7:
  \smallskip
  \begin{align*}
  & \{ \{ a{,}b{,}c{,}d \}{,}e{,}f{,}g \}  
  + \{ \{ a{,}b{,}f{,}e \}{,}d{,}c{,}g \}  
  + \{ \{ d{,}c{,}f{,}e \}{,}a{,}b{,}g \}  
  - \{ g{,}\{ b{,}a{,}d{,}c \}{,}f{,}e \}  
  \\
  &
  - \{ g{,}\{ b{,}a{,}e{,}f \}{,}c{,}d \}  
  - \{ g{,}\{ c{,}d{,}e{,}f \}{,}b{,}a \}  
  \equiv 0,
  \\[4pt]
    & \{ \{ a{,}b{,}c{,}d \}{,}e{,}f{,}g \}  
  - \{ \{ a{,}b{,}g{,}f \}{,}e{,}c{,}d \}  
  + \{ \{ a{,}b{,}d{,}c \}{,}e{,}g{,}f \}  
  - \{ \{ a{,}b{,}f{,}g \}{,}e{,}d{,}c \}  
  \\
  &
  + \{ \{ a{,}e{,}c{,}d \}{,}b{,}g{,}f \}  
  - \{ \{ a{,}e{,}g{,}f \}{,}b{,}d{,}c \}  
  + \{ \{ a{,}e{,}d{,}c \}{,}b{,}f{,}g \}  
  - \{ \{ a{,}e{,}f{,}g \}{,}b{,}c{,}d \}  
  \\
  &
  - \{ a{,}\{ b{,}c{,}d{,}e \}{,}f{,}g \}  
  + \{ a{,}\{ b{,}g{,}f{,}e \}{,}c{,}d \}  
  - \{ a{,}\{ b{,}d{,}c{,}e \}{,}g{,}f \}  
  + \{ a{,}\{ b{,}f{,}g{,}e \}{,}d{,}c \}  
  \equiv 0,
  \\[4pt]
  & \{ \{ a{,}b{,}c{,}d \}{,}e{,}f{,}g \}  
  - \{ \{ a{,}f{,}g{,}c \}{,}b{,}e{,}d \}  
  + \{ \{ c{,}b{,}a{,}d \}{,}e{,}g{,}f \}  
  + \{ \{ f{,}b{,}c{,}e \}{,}g{,}a{,}d \}  
  \\
  &
  - \{ \{ f{,}g{,}a{,}e \}{,}c{,}b{,}d \}  
  - \{ \{ f{,}g{,}a{,}d \}{,}b{,}c{,}e \}  
  - \{ \{ f{,}g{,}e{,}d \}{,}a{,}b{,}c \}  
  + \{ \{ g{,}b{,}a{,}e \}{,}f{,}c{,}d \}  
  \\
  &
  - \{ \{ g{,}f{,}c{,}e \}{,}a{,}b{,}d \}  
  - \{ \{ g{,}f{,}c{,}d \}{,}b{,}a{,}e \}  
  - \{ \{ g{,}f{,}e{,}d \}{,}c{,}b{,}a \}  
  + \{ \{ e{,}a{,}b{,}d \}{,}c{,}f{,}g \}  
  \\
  &
  + \{ \{ e{,}c{,}b{,}d \}{,}a{,}g{,}f \}  
  + \{ a{,}\{ b{,}c{,}g{,}f \}{,}e{,}d \}  
  + \{ c{,}\{ b{,}a{,}f{,}g \}{,}e{,}d \}  
  - \{ f{,}\{ b{,}c{,}e{,}g \}{,}a{,}d \}  
  \\
  &
  - \{ g{,}\{ b{,}a{,}e{,}f \}{,}c{,}d \}  
  + \{ e{,}\{ a{,}g{,}f{,}c \}{,}b{,}d \}  
  \equiv 0.
  \end{align*}
\end{theorem}

\begin{proof}
By symmetry, we need to consider only two association types in degree 7,
  \[
  \{ \{ -,-,-,- \},-,-,- \}, \qquad \{ -,\{ -,-,-,- \},-,- \},
  \]
each having $7!/2 = 2520$ distinct multilinear monomials, for a total of 5040.
A basis of $\mathrm{Quad}(7)$ consists of these monomials ordered by association type and then by lex order of the
permutation.
A basis of the group algebra $\mathbb{Q} S_7$, the multilinear subspace in the free associative algebra,
consists of the permutations in lex order.
We record the expansion of the identity monomial in each association type:
  \begin{align*}
  \{ \{ a,b,c,d \},e,f,g \}
  &=
  abcdefg + dcbaefg + gfeabcd + gfedcba,
  \\
  \{ a,\{ b,c,d,e \},f,g \}
  &=
  abcdefg + aedcbfg + gfbcdea + gfedcba.
  \end{align*}
The other expansions are obtained by permutation of the arguments,
since the expansion map $E \colon \mathrm{Quad}(7) \to \mathbb{Q} S_7$ is an $S_7$-module homomorphism.

Using modular arithmetic ($p = 101$) we initialize a $5040 \times 5040$ matrix in which the $(i,j)$ entry is the
coefficient of the $i$-th associative monomial in the expansion of the $j$-th quaternary monomial.
We compute the RCF and find that the rank is 2520, and hence the nullity is 2520.
As in the proof of Lemma \ref{tetrad4}, we find a basis of the nullspace;
every coefficient belongs to $\{ 0, 1, 2, 50, 51, 99, 100 \}$.
For each vector, we multiply by 2 if 50 or 51 is a coefficient, and do nothing otherwise.
Since the rows are coefficient vectors of polynomial identities, we can multiply by nonzero scalars to
obtain equivalent identities with simpler coefficients.
Reducing modulo $p$ using symmetric representatives, the coefficients belong to $\{ -4, -2, -1, 0, 1, 2, 4 \}$.
We regard these as integers, and sort the vectors by increasing Euclidean length.

The next step is to extract from the linear basis of 2520 vectors a much smaller set of generators for the nullspace
as an $S_7$-module.
Using modular arithmetic, we initialize a zero matrix consisting of upper and lower $5040 \times 5040$ blocks.
For each vector, regarded as the coefficients of a polynomial identity, we:
  \begin{itemize}
  \item
  check to see if the identity belongs to the row space of the matrix;
  \item
  if not, apply all permutations of the arguments to the identity,
  and store the results in the rows of the lower block;
  \item
  compute the RCF of the matrix; if the rank increases, record the identity as a generator
  (at this point, the lower block is again zero).
  \end{itemize}
Only four identities increase the rank, and one belongs to the submodule generated by the others,
which are independent (none is a consequence of the other two).
These are the three identities in the statement of this theorem.
\end{proof}

\begin{remark}
The identities of Theorem \ref{tetrad7} can be checked by hand.
For example, the expansion of the first identity into the free associative algebra produces
  \begin{align*}
  &
  abcdefg
  + gfeabcd
  + dcbaefg
  + gfedcba
  + abfedcg
  + gcdabfe
  \\
  &
  {}
  + efbadcg
  + gcdefba
  + dcfeabg
  + gbadcfe
  + efcdabg
  + gbaefcd
  \\
  &
  {}
  - gbadcfe
  - efbadcg
  - gcdabfe
  - efcdabg
  - gbaefcd
  - dcbaefg
  \\
  &
  {}
  - gfeabcd
  - dcfeabg
  - gcdefba
  - abcdefg
  - gfedcba
  - abfedcg
  = 0.
  \end{align*}
We need computer algebra to prove that these identities are a complete set of $S_7$-module generators
for the tetrad identities in degree 7.
\end{remark}

\begin{definition} \label{defjqs}
A \textbf{Jordan quadruple system} (or \textbf{JQS}) is a vector space $Q$ over a field $\mathbb{F}$
of characteristic 0 or $p > 7$
with a quadrilinear map $\{-,-,-,-\}\colon Q^4 \to Q$ satisfying the symmetry of
Lemma \ref{tetrad4} and the identities of Theorem \ref{tetrad7}.
\end{definition}

\begin{table}[h]
\begin{tabular}{lr|rrr|rrrr|r}
&\; & \multicolumn{3}{|c|}{symmetries} & \multicolumn{4}{|c|}{expansions}
\\
\midrule
$\lambda$ &\; $d_\lambda$ &\; rows &\; cols &\; symm &\; rows &\; cols &\; rank &\; null &\; new
\\
\midrule
  $7$      &   1 &   3 &   2 &   0 &   2 &   1 &   1 &   1 &   1 \\
  $61$     &   6 &  18 &  12 &   4 &  12 &   6 &   3 &   9 &   5 \\
  $52$     &  14 &  42 &  28 &  12 &  28 &  14 &   8 &  20 &   8 \\
  $51^2$   &  15 &  45 &  30 &  16 &  30 &  15 &   6 &  24 &   8 \\
  $43$     &  14 &  42 &  28 &  12 &  28 &  14 &   7 &  21 &   9 \\
  $421$    &  35 & 105 &  70 &  36 &  70 &  35 &  18 &  52 &  16 \\
  $41^3$   &  20 &  60 &  40 &  24 &  40 &  20 &  10 &  30 &   6 \\
  $3^21$   &  21 &  63 &  42 &  20 &  42 &  21 &   9 &  33 &  13 \\
  $32^2$   &  21 &  63 &  42 &  20 &  42 &  21 &  12 &  30 &  10 \\
  $321^2$  &  35 & 105 &  70 &  36 &  70 &  35 &  17 &  53 &  17 \\
  $31^4$   &  15 &  45 &  30 &  16 &  30 &  15 &   9 &  21 &   5 \\
  $2^31$   &  14 &  42 &  28 &  12 &  28 &  14 &   7 &  21 &   9 \\
  $2^21^3$ &  14 &  42 &  28 &  12 &  28 &  14 &   6 &  22 &  10 \\
  $21^5$   &   6 &  18 &  12 &   4 &  12 &   6 &   3 &   9 &   5 \\
  $1^7$    &   1 &   3 &   2 &   0 &   2 &   1 &   0 &   2 &   2 \\
\midrule
\end{tabular}
\caption{$S_7$-module multiplicities for the tetrad in degree 7}
\label{tetrad7table}
\end{table}

\begin{lemma} \label{tetrad7lemma}
Over a field of characteristic 0 or $p > 7$,
column ``new'' of Table \ref{tetrad7table} gives the multiplicity
of the $S_7$-module $[\lambda]$ in the kernel of the expansion map for each partition $\lambda$.
\end{lemma}

\begin{proof}
Table \ref{tetrad7table} was computed using the methods of \S\ref{subsectionrepmat}.
Columns 1 and 2 give the partitions $\lambda$ and the dimensions $d_\lambda$.
In degree 7 we consider only the symmetries of the association types, since there are no liftings
of identities from lower degrees.
There are two association types each with one symmetry:
  \begin{align*}
  &
  \{ \{ a_1, a_2, a_3, a_4 \}, a_5, a_6, a_7 \}
  -
  \{ \{ a_4, a_3, a_2, a_1 \}, a_5, a_6, a_7 \} \equiv 0,
  \\
  &
  \{ a_1, \{ a_2, a_3, a_4, a_5 \}, a_6, a_7 \}
  -
  \{ a_1, \{ a_5, a_4, a_3, a_2 \}, a_6, a_7 \} \equiv 0.
  \qquad
  \end{align*}
Under ``symmetries'', columns 3 and 4 (rows, cols) give the size of the representation matrix for the symmetries,
and column 5 (symm) gives its rank.
Under ``expansions'', columns 6 and 7 give the size of the expansion matrix,
column 8 gives its rank, and column 9 gives the nullity of its \emph{transpose}.
(For an explanation of using the transpose, see \cite[pp.~446-447]{BM}.)
Column ``new'' is the multiplicity of the module $[\lambda]$ (null minus symm) in the kernel of the expansion map.
\end{proof}

\begin{remark} \label{tetradnonlinear}
The data in Table \ref{tetrad7table} show that there is no point in searching for possibly simpler
nonlinear identities whose linearizations are equivalent to the identities of Theorem
\ref{tetrad7}.
If $I \equiv 0$ is a nonlinear identity of degree 7,
then every monomial of $I$ contains $k \ge 2$ occurrences of some variable, say $a$.
In $L(I)$, the linearization of $I$, the variable $a$ is replaced by $k$ distinct variables
$a_1, \dots, a_k$ and $L(I)$ is invariant under all permutations of these variables.
To obtain the consequences of $L(I)$ in partition $1^7$,
we evaluate the alternating sum over all permutations of the variables in $L(I)$, but this is clearly 0,
contradicting the multiplicity 2 for the module $[1^7]$.
\end{remark}


\section{Special identities in degree 10 for the tetrad} \label{sectiondegree10}

In this section we show that there are identities for the tetrad
which do not follow from the defining identities for JQS.
(We have not included these identities in Definition \ref{defjqs}, so that the operad for JQS is
quadratic and admits a Koszul dual.)

\begin{definition} \label{defjqsspecial}
Let $A$ be an associative algebra and let $\mathrm{JQS}(A)$ be the underlying vector space of $A$
with the tetrad as the operation.
If $Q$ is a JQS, we call $Q$ \textbf{special} if there is an embedding $Q \hookrightarrow \mathrm{JQS}(A)$
for some $A$; otherwise we call $Q$ \textbf{exceptional}.
If $I \equiv 0$ is a polynomial identity which is satisfied by the tetrad in every associative algebra
but which does not follow from the defining identities for JQS, then we call $I$ a \textbf{special identity}
for the tetrad.
\end{definition}

\begin{openproblem}
Do there exist exceptional JQS?
\end{openproblem}

\begin{theorem} \label{tetrad10}
For the tetrad, and for each partition $\lambda$ of 10, columns ``symm'', ``symmlift'', ``null'' and ``new''
in the corresponding row of Table \ref{tetrad10table} contain respectively the multiplicity of $[\lambda]$
in the modules $\mathrm{Symm}(10)$, $\mathrm{Old}(10)$, $\mathrm{All}(10)$ and $\mathrm{New}(10)$.
\end{theorem}

\begin{table} \small
\[
\begin{array}{r|lr|rrrr|rrrr|r}
&\; &\; & \multicolumn{4}{|c|}{\text{symmetries and liftings}} & \multicolumn{4}{|c|}{\text{expansions}}
\\
\midrule
\# & \;\lambda &\; d_\lambda &\;
\text{rows} &\; \text{cols} &\; \text{symm} &\; \text{symmlift}  &\;
\text{rows} &\; \text{cols} &\; \text{rank} &\; \text{null} &\; \text{new}
\\
\midrule
   1  &  10             &    1&     9  &     8  &     0  &     7  &     8  &     1  &     1  &     7  &  \cdot  \\
   2  &  9 1          &    9&    81  &    72  &    30  &    68  &    72  &     9  &     4  &    68  &  \cdot  \\
   3  &  8 2          &   35&   315  &   280  &   140  &   260  &   280  &    35  &    20  &   260  &  \cdot  \\
   4  &  8 1^2        &   36&   324  &   288  &   166  &   272  &   288  &    36  &    16  &   272  &  \cdot  \\
   5  &  7 3          &   75&   675  &   600  &   342  &   565  &   600  &    75  &    35  &   565  &  \cdot  \\
   6  &  7 2 1      &  160&  1440  &  1280  &   784  &  1200  &  1280  &   160  &    80  &  1200  &  \cdot  \\
   7  &  7 1^3        &   84&   756  &   672  &   442  &   628  &   672  &    84  &    44  &   628  &  \cdot  \\
   8  &  6 4          &   90&   810  &   720  &   416  &   670  &   720  &    90  &    50  &   670  &  \cdot  \\
   9  &  6 3 1      &  315&  2835  &  2520  &  1594  &  2365  &  2520  &   315  &   155  &  2365  &  \cdot  \\
  10  &  6 2^2        &  225&  2025  &  1800  &  1150  &  1680  &  1800  &   225  &   120  &  1680  &  \cdot  \\
  11  &  6 2 1^2    &  350&  3150  &  2800  &  1878  &  2630  &  2800  &   350  &   170  &  2630  &  \cdot  \\
  12  &  6 1^4        &  126&  1134  &  1008  &   704  &   941  &  1008  &   126  &    66  &   942  &      1  \\
  13  &  5^2            &   42&   378  &   336  &   204  &   320  &   336  &    42  &    16  &   320  &  \cdot  \\
  14  &  5 4 1      &  288&  2592  &  2304  &  1456  &  2160  &  2304  &   288  &   144  &  2160  &  \cdot  \\
  15  &  5 3 2      &  450&  4050  &  3600  &  2330  &  3380  &  3600  &   450  &   220  &  3380  &  \cdot  \\
  16  &  5 3 1^2    &  567&  5103  &  4536  &  3008  &  4244  &  4536  &   567  &   291  &  4245  &      1  \\
  17  &  5 2^2 1      &  525&  4725  &  4200  &  2808  &  3939  &  4200  &   525  &   260  &  3940  &      1  \\
  18  &  5 2 1^3    &  448&  4032  &  3584  &  2464  &  3357  &  3584  &   448  &   224  &  3360  &      3  \\
  19  &  5 1^5        &  126&  1134  &  1008  &   708  &   945  &  1008  &   126  &    60  &   948  &      3  \\
  20  &  4^2 2          &  252&  2268  &  2016  &  1282  &  1880  &  2016  &   252  &   136  &  1880  &  \cdot  \\
  21  &  4^2 1^2        &  300&  2700  &  2400  &  1582  &  2260  &  2400  &   300  &   140  &  2260  &  \cdot  \\
  22  &  4 3^2        &  210&  1890  &  1680  &  1092  &  1580  &  1680  &   210  &   100  &  1580  &  \cdot  \\
  23  &  4 3 2 1  &  768&  6912  &  6144  &  4032  &  5760  &  6144  &   768  &   384  &  5760  &  \cdot  \\
  24  &  4 3 1^3    &  525&  4725  &  4200  &  2802  &  3933  &  4200  &   525  &   265  &  3935  &      2  \\
  25  &  4 2^3        &  300&  2700  &  2400  &  1562  &  2239  &  2400  &   300  &   160  &  2240  &      1  \\
  26  &  4 2^2 1^2    &  567&  5103  &  4536  &  3022  &  4260  &  4536  &   567  &   276  &  4260  &  \cdot  \\
  27  &  4 2 1^4    &  350&  3150  &  2800  &  1870  &  2616  &  2800  &   350  &   180  &  2620  &      4  \\
  28  &  4 1^6        &   84&   756  &   672  &   446  &   631  &   672  &    84  &    40  &   632  &      1  \\
  29  &  3^3 1          &  210&  1890  &  1680  &  1080  &  1569  &  1680  &   210  &   110  &  1570  &      1  \\
  30  &  3^2 2^2        &  252&  2268  &  2016  &  1302  &  1900  &  2016  &   252  &   116  &  1900  &  \cdot  \\
  31  &  3^2 2 1^2    &  450&  4050  &  3600  &  2322  &  3368  &  3600  &   450  &   230  &  3370  &      2  \\
  32  &  3^2 1^4        &  225&  2025  &  1800  &  1164  &  1693  &  1800  &   225  &   105  &  1695  &      2  \\
  33  &  3 2^3 1      &  288&  2592  &  2304  &  1456  &  2159  &  2304  &   288  &   144  &  2160  &      1  \\
  34  &  3 2^2 1^3    &  315&  2835  &  2520  &  1588  &  2358  &  2520  &   315  &   160  &  2360  &      2  \\
  35  &  3 2 1^5    &  160&  1440  &  1280  &   784  &  1198  &  1280  &   160  &    80  &  1200  &      2  \\
  36  &  3 1^7        &   36&   324  &   288  &   162  &   268  &   288  &    36  &    20  &   268  &  \cdot  \\
  37  &  2^5            &   42&   378  &   336  &   192  &   309  &   336  &    42  &    26  &   310  &      1  \\
  38  &  2^4 1^2        &   90&   810  &   720  &   428  &   680  &   720  &    90  &    40  &   680  &  \cdot  \\
  39  &  2^3 1^4        &   75&   675  &   600  &   336  &   559  &   600  &    75  &    40  &   560  &      1  \\
  40  &  2^2 1^6        &   35&   315  &   280  &   146  &   265  &   280  &    35  &    15  &   265  &  \cdot  \\
  41  &  2 1^8        &    9&    81  &    72  &    28  &    67  &    72  &     9  &     5  &    67  &  \cdot  \\
  42  &  1^{10}         &    1&     9  &     8  &     2  &     8  &     8  &     1  &     0  &     8  &  \cdot  \\
\midrule
\end{array}
\]
\caption{$S_{10}$-module multiplicities for the tetrad in degree 10}
\label{tetrad10table}
\end{table}

\begin{proof}
The symmetry in degree 4 implies that every association type for a quadrilinear operation in degree 10
is equivalent to one of the following:
  \begin{equation}
  \label{types10}
  \left\{ \quad
  \begin{array}{cc}
  \{ \{ \{ -, -, -, - \}, -, -, - \}, -, -, - \}, &\qquad
  \{ \{ -, \{ -, -, -, - \}, -, - \}, -, -, - \},
  \\[3pt]
  \{ -, \{ \{ -, -, -, - \}, -, -, - \}, -, - \}, &\qquad
  \{ -, \{ -, \{ -, -, -, - \}, -, - \}, -, - \},
  \\[3pt]
  \{ \{ -, -, -, - \}, \{ -, -, -, - \}, -, - \}, &\qquad
  \{ \{ -, -, -, - \}, -, \{ -, -, -, - \}, - \},
  \\[3pt]
  \{ \{ -, -, -, - \}, -, -, \{ -, -, -, - \} \}, &\qquad
  \{ -, \{ -, -, -, - \}, \{ -, -, -, - \}, - \}.
  \end{array}
  \right.
  \end{equation}
These 8 types have respectively 1, 1, 1, 1, 2, 2, 3, 3 symmetries $\iota - \tau \equiv 0$ where
$\iota$ is the identity monomial and $\tau$ is a monomial with a permutation of order 2 of the variables.
We list the corresponding monomials $\tau$ for each type:
  \begin{alignat*}{4}
  \{ \{ \{ d c b a \} e f g \} h i j \}, \quad
  &
  \{ \{ a \{ e d c b \} f g \} h i j \}, \quad
  &
  \{ a \{ \{ e d c b \} f g h \} i j \}, \quad
  &
  \{ a \{ b \{ f e d c \} g h \} i j \},
  \\
  \{ \{ d c b a \} \{ e f g h \} i j \}, \quad
  &
  \{ \{ a b c d \} \{ h g f e \} i j \}, \quad
  &
  \{ \{ d c b a \} e \{ f g h i \} j \}, \quad
  &
  \{ \{ a b c d \} e \{ i h g f \} j \},
  \\
  \{ \{ d c b a \} e f \{ g h i j \} \}, \quad
  &
  \{ \{ a b c d \} e f \{ j i h g \} \}, \quad
  &
  \{ \{ g h i j \} f e \{ a b c d \} \}, \quad
  &
  \{ a \{ e d c b \} \{ f g h i \} j \},
  \\
  \{ a \{ b c d e \} \{ i h g f \} j \}, \quad
  &
  \{ j \{ f g h i \} \{ b c d e \} a \}.
  \end{alignat*}
For each association type, each symmetry reduces the number of multilinear monomials by a factor of 2,
so the total number of monomials is
  \[
  10! \left( \tfrac12 + \tfrac12 + \tfrac12 + \tfrac12 + \tfrac14 + \tfrac14 + \tfrac18 + \tfrac18 \right)
  =
  9979200.
  \]
This number is so large that we must decompose the computation into smaller pieces using the representation
theory of $S_{10}$.
(To do the following computation without representation theory would require a $3628800 \times 9979200$ matrix.)

The symmetry implies that each identity $I(a,b,c,d,e,f,g)$ from Theorem \ref{tetrad7} produces 9 liftings
in degree 10:
  \begin{align*}
  &
  I(\{a{,}h{,}i{,}j\}{,}b{,}c{,}d{,}e{,}f{,}g), \quad
  I(a{,}\{b{,}h{,}i{,}j\}{,}c{,}d{,}e{,}f{,}g), \quad
  \dots, \quad
  I(a{,}b{,}c{,}d{,}e{,}f{,}\{g{,}h{,}i{,}j\}),
  \\
  &
  \{I(a{,}b{,}c{,}d{,}e{,}f{,}g){,}h{,}i{,}j\}, \quad
  \{h,I(a{,}b{,}c{,}d{,}e{,}f{,}g){,}i{,}j\}.
  \end{align*}
For each partition $\lambda$ with module $[\lambda]$ of dimension $d_\lambda$,
we construct a matrix $M$ of size $9 d_\lambda \times 8 d_\lambda$
(``rows'' and ``cols'' under ``symmetries and liftings'')
consisting of $d_\lambda \times d_\lambda$ blocks.
To process an identity, we compute the representation matrices of its terms in each association type \cite{BP2011},
store these matrices in the last row of blocks, and compute the RCF of $M$,
so the last row of blocks becomes zero.
In this way, we process the symmetries of the association types and compute the rank of $M$ (column ``symm'').
We retain these results in $M$ and process the liftings; the rank of $M$ increases (column ``symmlift'')
to reach the multiplicity of $[\lambda]$ in the module $\mathrm{Old}(10)$.
When we are done, the rows of $\mathrm{RCF}(M)$ form a canonical set of generators for
the isotypic component of $[\lambda]$ in $\mathrm{Old}(10)$.

For each partition $\lambda$, we construct a matrix $X$ of size $8 d_\lambda \times d_\lambda$ (``rows'' and ``cols''
under ``expansions'') consisting of $d_\lambda \times d_\lambda$ blocks;
the $i$-th block contains the representation matrix for the terms of the expansion of the identity monomial in
the $i$-th association type.
We compute the RCF of the transpose $X^t$ and find its rank (column ``rank'') and nullity (column ``null);
we then extract a basis for the nullspace by setting the free variables to the standard basis vectors
and solving for the leading variables.
We put the nullspace basis into the rows of a matrix $N$ and compute the RCF.
The rows of $\mathrm{RCF}(N)$ form a canonical set of generators for the isotypic component of $[\lambda]$
in $\mathrm{All}(10)$.

If the multiplicity of $[\lambda]$ in $\mathrm{Old}(10)$ coincides with its multiplicity in $\mathrm{All}(10)$,
then there are no new identities for partition $\lambda$.
(In this case, we check the consistency of the computations by verifying that the two submodules are equal,
which amounts to verifying that the two matrices in RCF are equal.)
Otherwise, there are new identities in degree 10 for the tetrad corresponding to partition $\lambda$.
\end{proof}

\begin{openproblem}
Determine a minimal set of multilinear identities in degree 10 for the tetrad which generate $\mathrm{New}(10)$
as an $S_{10}$-module.
\end{openproblem}

In the rest of this section we present some nonlinear special identities in degree 10 for the tetrad.
These identities correspond to the partitions $\lambda = s 1^t$ ($s+t=10$) for which column ``new'' in
Table \ref{tetrad10table} is nonzero, namely $\lambda = 61^4, 51^5, 41^6$.
The corresponding Young diagrams have a long first row and a long tail.
For these $\lambda$, the structure theory of $\field S_{10}$ allows us to assume that the identity is
a symmetric function of the variables in the first row and an alternating function of the variables
in the tail, and this reduces the computational problem to a manageable size.
We first recall the notion of linearization matrices, called operators in \cite[\S 2]{BP2009b}.

\begin{definition}
    Let $n = n_1 + \dots + n_k$ ($n \ge n_1 \ge \dots \ge n_k \ge 1$) be a partition
    and $\mu = \mu(a_1^{n_1}, \dots, a_k^{n_k})$ a quaternary monomial of multidegree $(n_1, \dots, n_k)$.
    For $i=1, \dots, k$ set $N_i= \{ n_1 + \dots + n_{i-1}+1, \dots, n_1 + \dots + n_i \}$
    so that $N_1 \cup N_2 \cup \cdots \cup N_k$ is a set partition of $\{ 1, \dots, n\}$ with $| N_i | = n_i$.
    The \textbf{basic linearization} of $\mu$ is the multilinear monomial $\ell(\mu)$ obtained from $\mu$
    by first replacing $a_i$ by $a_i'$ for $i=1,\dots,k$ and then replacing the $n_i$ occurrences of $a_i'$
    by $a_j$ for $j \in N_i$ with $j$ increasing from left to right.
    The \textbf{symmetric} and \textbf{alternating linearizations} of $\mu$ are
    \begin{align*}
        L^+(\mu) &= \sum_{\sigma_1 \in P_1} \cdots \sum_{\sigma_k \in P_k} \sigma_1 \cdots \sigma_k \ell(\mu),
        \\
        L^-(\mu) &= \sum_{\sigma_1 \in P_1} \cdots \sum_{\sigma_k \in P_k}
                 \epsilon_i(\sigma_1) \cdots \epsilon_k(\sigma_k) \sigma_1 \cdots \sigma_k \ell(\mu),
    \end{align*}
    where $P_i$ is the symmetric group on $N_i$ and $\epsilon_i \colon P_i \to \{ \pm 1 \}$ is the sign.
\end{definition}

\begin{remark}
    Over a field of characteristic 0 or $p > n$, the identity $I$ holds if and only if $L^+(I)$ holds;
    the same is not true for $L^-(I)$.
\end{remark}

\begin{definition} \label{deflinops}
    Let $\lambda$ be another partition of $n$ and let
    $R_\lambda \colon \mathbb{Q} S_n \to M_d(\mathbb{Q})$ be
    the corresponding irreducible representation of $S_n$ of dimension $d_\lambda$.
    The \textbf{symmetric} and \textbf{alternating linearization matrices} for partitions $n_1, \dots n_k$ and $\lambda$
    are
    \[
    L^+_{n_1, \dots, n_k; \lambda} = \sum_{\sigma \in P_1 \times \cdots \times P_k} R_\lambda(\sigma),
    \qquad
    L^-_{n_1, \dots, n_k; \lambda} = \sum_{\sigma \in P_1 \times \cdots \times P_k} \epsilon(\sigma) R_\lambda(\sigma).
    \]
\end{definition}

Linearization matrices make it much easier to compute the representation matrix for the linearization of
a nonlinear monomial $\mu$: we can use the basic linearization $\ell(\mu)$ of the nonlinear monomial
instead of using all the terms in $L^+( \mu )$.

\begin{lemma} \label{lemmalinops}
For all partitions $n_1, \dots, n_k$ and $\lambda$ and all monomials $\mu$, we have
  \[
  R_\lambda( L^+( \mu ) )
  =
  L^+_{n_1, \dots, n_k; \lambda} \cdot R_\lambda( \ell(\mu) ),
  \qquad
  R_\lambda( L^-( \mu ) )
  =
  L^-_{n_1, \dots, n_k; \lambda} \cdot R_\lambda( \ell(\mu) ).
  \]
\end{lemma}

\begin{proof}
These equations follow directly from the definitions, the linearity of the maps, and the homomorphism property of
a representation.
\end{proof}

\begin{figure}
  \begin{align*}
  \sum_{\sigma \in S_4}
  \epsilon(\sigma)
  &\Big( \;
  \{ \{ \{ a, a, a, b^\sigma \}, a, a, c^\sigma \}, d^\sigma, a, e^\sigma \}
  - \{ \{ \{ a, a, a, b^\sigma \}, a, a, c^\sigma \}, d^\sigma, e^\sigma, a \}
  \\[-8pt]
  &{}
  + \{ \{ \{ a, a, b^\sigma, a \}, c^\sigma, a, a \}, d^\sigma, e^\sigma, a \}
  + \{ \{ \{ a, a, b^\sigma, a \}, c^\sigma, a, d^\sigma \}, a, e^\sigma, a \}
  \\[2pt]
  &{}
  + \{ \{ \{ a, b^\sigma, a, c^\sigma \}, a, d^\sigma, a \}, a, a, e^\sigma \}
  - \{ \{ \{ a, b^\sigma, a, c^\sigma \}, a, d^\sigma, a \}, a, e^\sigma, a \}
  \\[2pt]
  &{}
  - \{ \{ a, \{ a, a, b^\sigma, a \}, a, c^\sigma \}, d^\sigma, a, e^\sigma \}
  - \{ \{ a, \{ a, b^\sigma, a, c^\sigma \}, a, d^\sigma \}, a, e^\sigma, a \}
  \\[2pt]
  &{}
  + \{ a, \{ \{ a, a, b^\sigma, a \}, a, c^\sigma, d^\sigma \}, a, e^\sigma \}
  - \{ a, \{ \{ a, a, b^\sigma, c^\sigma \}, d^\sigma, a, a \}, a, e^\sigma \}
  \\[-2pt]
  &{}
  + \{ a, \{ \{ a, a, b^\sigma, c^\sigma \}, d^\sigma, a, a \}, e^\sigma, a \}
  + \{ a, \{ \{ a, b^\sigma, a, c^\sigma \}, a, d^\sigma, a \}, e^\sigma, a \}
  \; \Big)
  \equiv 0
  \\
  \sum_{\sigma \in S_5}
  \epsilon(\sigma)
  &\Big( \;
  \{ \{ \{ a, a, b^\sigma, c^\sigma \}, a, d^\sigma, a \}, a, e^\sigma, f^\sigma \}
  - \{ \{ \{ a, a, b^\sigma, c^\sigma \}, a, d^\sigma, a \}, e^\sigma, a, f^\sigma \}
  \\[-8pt]
  &{}
  - \{ \{ \{ a, a, b^\sigma, c^\sigma \}, d^\sigma, a, a \}, a, e^\sigma, f^\sigma \}
  + \{ \{ \{ a, a, b^\sigma, c^\sigma \}, d^\sigma, a, a \}, e^\sigma, a, f^\sigma \}
  \\[2pt]
  &{}
  - \{ \{ \{ a, b^\sigma, a, c^\sigma \}, a, d^\sigma, a \}, a, e^\sigma, f^\sigma \}
  + \{ \{ \{ a, b^\sigma, a, c^\sigma \}, a, d^\sigma, a \}, e^\sigma, a, f^\sigma \}
  \\[2pt]
  &{}
  + \{ \{ \{ a, b^\sigma, a, c^\sigma \}, d^\sigma, a, a \}, a, e^\sigma, f^\sigma \}
  - \{ \{ \{ a, b^\sigma, a, c^\sigma \}, d^\sigma, a, a \}, e^\sigma, a, f^\sigma \}
  \\[2pt]
  &{}
  - \{ \{ \{ b^\sigma, a, c^\sigma, d^\sigma \}, a, e^\sigma, a \}, a, f^\sigma, a \}
  - \{ \{ \{ b^\sigma, a, c^\sigma, d^\sigma \}, e^\sigma, a, a \}, f^\sigma, a, a \}
  \\[-2pt]
  &{}
  - \{ \{ \{ b^\sigma, c^\sigma, a, d^\sigma \}, a, e^\sigma, a \}, f^\sigma, a, a \}
  - \{ \{ \{ b^\sigma, c^\sigma, a, d^\sigma \}, e^\sigma, a, a \}, a, f^\sigma, a \}
  \; \Big)
  \equiv 0
  \\
  \sum_{\sigma \in S_5}
  \epsilon(\sigma)
  &\Big( \;
  \{ \{ \{ a, a, a, b^\sigma \}, c^\sigma, d^\sigma, a \}, e^\sigma, a, f^\sigma \}
  - \{ \{ \{ a, a, a, b^\sigma \}, c^\sigma, d^\sigma, a \}, e^\sigma, f^\sigma, a \}
  \\[-8pt]
  &{}
  + \{ \{ \{ a, b^\sigma, a, c^\sigma \}, d^\sigma, e^\sigma, a \}, a, a, f^\sigma \}
  - \{ \{ \{ a, b^\sigma, a, c^\sigma \}, d^\sigma, e^\sigma, a \}, a, f^\sigma, a \}
  \\[2pt]
  &{}
  - \{ \{ a, \{ a, a, b^\sigma, c^\sigma \}, a, d^\sigma \}, e^\sigma, f^\sigma, a \}
  - \{ \{ a, \{ a, a, b^\sigma, c^\sigma \}, d^\sigma, a \}, e^\sigma, a, f^\sigma \}
  \\[2pt]
  &{}
  + \{ \{ a, \{ a, a, b^\sigma, c^\sigma \}, d^\sigma, a \}, e^\sigma, f^\sigma, a \}
  + \{ \{ a, \{ a, b^\sigma, a, c^\sigma \}, a, d^\sigma \}, e^\sigma, f^\sigma, a \}
  \\[2pt]
  &{}
  - \{ \{ a, \{ b^\sigma, a, c^\sigma, d^\sigma \}, e^\sigma, a \}, a, a, f^\sigma \}
  - \{ \{ a, \{ b^\sigma, c^\sigma, a, d^\sigma \}, e^\sigma, a \}, a, f^\sigma, a \}
  \\[-2pt]
  &{}
  + \{ a, \{ \{ a, a, b^\sigma, c^\sigma \}, a, d^\sigma, e^\sigma \}, f^\sigma, a \}
  - \{ a, \{ \{ a, b^\sigma, a, c^\sigma \}, a, d^\sigma, e^\sigma \}, f^\sigma, a \}
  \; \Big)
  \equiv 0
  \\
  \sum_{\sigma \in S_5}
  \epsilon(\sigma)
  &\Big( \;
  \{ \{ \{ a, a, b^\sigma, a \}, c^\sigma, a, d^\sigma \}, e^\sigma, f^\sigma, a \}
  - \{ \{ \{ a, b^\sigma, a, c^\sigma \}, a, a, d^\sigma \}, e^\sigma, f^\sigma, a \}
  \\[-8pt]
  &{}
  - \{ \{ \{ a, b^\sigma, a, c^\sigma \}, a, d^\sigma, a \}, e^\sigma, f^\sigma, a \}
  + \{ \{ \{ a, b^\sigma, a, c^\sigma \}, d^\sigma, a, e^\sigma \}, a, a, f^\sigma \}
  \\[2pt]
  &{}
  + \{ \{ \{ a, b^\sigma, a, c^\sigma \}, d^\sigma, e^\sigma, a \}, a, a, f^\sigma \}
  + \{ \{ a, \{ a, a, b^\sigma, c^\sigma \}, d^\sigma, a \}, e^\sigma, a, f^\sigma \}
  \\[2pt]
  &{}
  + \{ \{ a, \{ a, b^\sigma, a, c^\sigma \}, a, d^\sigma \}, e^\sigma, a, f^\sigma \}
  + \{ \{ a, \{ a, b^\sigma, a, c^\sigma \}, d^\sigma, a \}, e^\sigma, f^\sigma, a \}
  \\[2pt]
  &{}
  - \{ \{ a, \{ b^\sigma, a, c^\sigma, d^\sigma \}, a, e^\sigma \}, a, a, f^\sigma \}
  - \{ \{ a, \{ b^\sigma, c^\sigma, a, d^\sigma \}, a, e^\sigma \}, a, f^\sigma, a \}
  \\[2pt]
  &{}
  + \{ \{ b^\sigma, \{ a, a, c^\sigma, a \}, d^\sigma, a \}, e^\sigma, a, f^\sigma \}
  - \{ a, \{ \{ a, a, b^\sigma, c^\sigma \}, d^\sigma, a, e^\sigma \}, a, f^\sigma \}
  \\[2pt]
  &{}
  - \{ a, \{ \{ a, b^\sigma, a, c^\sigma \}, a, d^\sigma, e^\sigma \}, a, f^\sigma \}
  - \{ a, \{ \{ a, b^\sigma, a, c^\sigma \}, a, d^\sigma, e^\sigma \}, f^\sigma, a \}
  \\[-2pt]
  &{}
  - \{ a, \{ \{ a, b^\sigma, a, c^\sigma \}, d^\sigma, a, e^\sigma \}, f^\sigma, a \}
  \; \Big)
  \equiv 0
  \\
  \sum_{\sigma \in S_6}
  \epsilon(\sigma)
  &\Big( \;
  \{ \{ \{ b^\sigma, a, c^\sigma, d^\sigma \}, a, a, e^\sigma \}, a, f^\sigma, g^\sigma \}
  - \{ \{ \{ b^\sigma, a, c^\sigma, d^\sigma \}, a, a, e^\sigma \}, f^\sigma, a, g^\sigma \}
  \\[-8pt]
  &{}
  + \{ \{ \{ b^\sigma, a, c^\sigma, d^\sigma \}, e^\sigma, a, f^\sigma \}, a, a, g^\sigma \}
  - \{ \{ \{ b^\sigma, c^\sigma, a, d^\sigma \}, a, a, e^\sigma \}, a, f^\sigma, g^\sigma \}
  \\[-2pt]
  &{}
  + \{ \{ \{ b^\sigma, c^\sigma, a, d^\sigma \}, a, a, e^\sigma \}, f^\sigma, a, g^\sigma \}
  + \{ \{ \{ b^\sigma, c^\sigma, a, d^\sigma \}, a, e^\sigma, f^\sigma \}, a, a, g^\sigma \}
  \; \Big)
  \equiv 0
  \end{align*}
\caption{Some special identities for the tetrad in degree 10}
\label{special10}
\end{figure}

\begin{theorem}
The multihomogeneous identities in Figure \ref{special10} are satisfied by the tetrad in every associative
algebra but are not consequences of the defining identities for JQS.
The sums are over all permutations $\sigma$ of the non-repeated variables and $\epsilon(\sigma)$ is
the sign.
\end{theorem}

\begin{proof}
The Young tableaux for partitions $\lambda = 61^4, 51^5, 41^6$ are
  
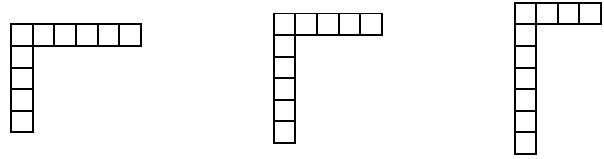
\[
  \begin{array}{c} \scriptsize
  \yng(6,1,1,1,1)
  \end{array}
  \qquad\qquad
  \begin{array}{c} \scriptsize
  \yng(5,1,1,1,1,1)
  \end{array}
  \qquad\qquad
  \begin{array}{c} \scriptsize
  \yng(4,1,1,1,1,1,1)
  \end{array}
  \]
For each $\lambda$, the primitive idempotents in the Young basis of the two-sided ideal in
the group algebra correspond to standard tableaux; each idempotent is the product of the symmetric sum over
row permutations and the alternating sum over column permutations.
Hence an identity corresponding to a partition $\lambda$ with first row of size $s$ and tail of size $t$
is symmetric in the first $s$ variables and alternating in the last $t$ variables.
This allows us to reduce the number of monomials by considering only those in which the non-repeated variables
appear in increasing order, and identifying these monomials with the corresponding alternating sums.
We write in detail the proof for $\lambda = 61^4$; the others are similar.

Let $J$ be the multihomogeneous subspace with variables $a^6bcde$ in the free quaternary algebra with one operation
$\{-,-,-,-\}$ satisfying symmetry in degree 4.
For the eight association types \eqref{types10} we have respectively 134, 134, 134, 134, 86, 86, 49, 52 monomials,
and so $\dim J = 809$ (recall that the non-repeated variables appear in increasing order).
Let $A$ be the corresponding multihomogeneous subspace in the free associative algebra; we have $\dim A = 5040$.
We construct the $809 \times 5040$ expansion matrix $E$ in which the $(i,j)$ entry is the (integer) coefficient of
the $j$-th associative monomial in the expansion of the $i$-th quaternary monomial.
Using modular arithmetic, we compute $\mathrm{RCF}(E)$ and obtain rank 110 and nullity 699; the entries of the RCF
are 0, $\pm 1$, $\pm 2$, $\pm 4$.
We identify the columns which contain leading 1s; these columns $j_1, \dots, j_{110}$ form a basis of the column space.
Using integer arithmetic, we combine these columns of $E$ into a much smaller $809 \times 110$ matrix $E'$.
We compute $H = \mathrm{HNF}(E')$ and confirm that the rank and nullity have not changed in characteristic 0;
this computation also provides an $809 \times 809$ unimodular matrix $U$ such that $U E' = H$.
The last 699 rows of $U$ form a lattice basis for the integer nullspace of $E'$; they have already been reduced using
the LLL algorithm, and the entries are 0, $\pm 1$, $\pm 2$.
Using formula \eqref{basissize}, this basis has size $\approx 607$.
To get identities as simple as possible, we apply the LLL algorithm again, with standard parameter $3/4$;
this new basis has size $\approx 537$.
We sort these 699 vectors by increasing length.

We proceed as in the proof of Theorem \ref{tetrad10} using modular arithmetic ($p = 101$) but considering
only partition $\lambda = 61^4$.
The matrix has size $1134 \times 1008$.
We first compute the rank of the symmetries (704) and the symmetries with the liftings (941).
The usual next step would be to process the linearizations of the 699 nonlinear identities obtained in
the previous paragraph; this however produces multilinear identities with very large numbers of terms.
In order to reduce the size of the computation, we use the equations of Lemma \ref{lemmalinops}.
We require a combination of symmetric linearization over the first 6 variables (in the first row
of the Young diagram) with alternating linearization over the last 4 variables (in the tail).
This involves a slight modification of the linearization matrices in Definition \ref{deflinops}.
Our combined linearization matrix is the sum of $6!4! = 17280$ terms, but it factors as the product of
two partial linearization matrices with respectively
$6!$ terms (permutations $\sigma \in S_6$ acting on $\{1,\dots,6\}$)
and $4!$ terms (permutations $\tau \in S_4$ acting on $\{7,\dots,10\}$),
for a total of only $6! + 4!= 744$ terms:
  \[
  \sum_{\sigma \in S_6, \tau \in S_4}
  \!\!\!
  \epsilon(\tau)
  R_\lambda( \sigma \tau )
  =
  \sum_{\sigma \in S_6, \tau \in S_4}
  \!\!\!
  \epsilon(\tau)
  R_\lambda( \sigma )
  R_\lambda( \tau )
  =
  \sum_{\sigma \in S_6}
  R_\lambda( \sigma )
  \cdot
  \sum_{\tau \in S_4}
  \epsilon(\tau)
  R_\lambda( \tau ).
  \]
This $126 \times 126$ matrix is surprisingly simple: it has rank 1 and 21 nonzero entries.
Since the rank of all identities for $\lambda = 61^4$ is 942, we expect exactly one new identity.
After processing the 699 nonlinear identities, we find that only one increases the rank.
This is the identity in the statement of the theorem.
\end{proof}


\section{Four families of finite dimensional Jordan quadruple systems;\\
examples of universal associative envelopes}
\label{sectionUtetrad}

In this section we consider some small examples of special JQS,
and determine the structure of their universal associative envelopes.
We first define four infinite families $A_n$, $B_n$, $C_{pqr}$, $D_{pqq}$ of finite dimensional JQS
which are subspaces of matrix algebras which are closed under the tetrad.
  \begin{itemize}
  \item
  $A_n$ consists of all $n \times n$ matrices; the dimension is $n^2$.
  \item
  $B_n$ consists of all $n \times n$ symmetric matrices; the dimension is $\tfrac12 n(n{+}1)$.
  \end{itemize}
For the remaining systems, we recall the classification \cite{Carlsson} of simple associative $n$-ary algebras
in the case $n = 4$.
We choose integers $p, q, r \ge 1$ and consider block matrices of size $p+q+r$,
where $M_{pq}$ is a matrix of size $p \times q$:
  \begin{equation}
  \label{pqrmatrix}
  \left[
  \begin{array}{ccc}
  0      & 0      & M_{pr} \\
  M_{qp} & 0      & 0      \\
  0      & M_{rq} & 0
  \end{array}
  \right]
  \end{equation}
This space is closed under the associative quadruple product of matrices.
  \begin{itemize}
  \item
  $C_{pqr}$ ($p \ge q, r$) consists of all matrices \eqref{pqrmatrix}; the dimension is $pq+qr+rp$.
  \item
  $D_{pq}$ consists of all matrices \eqref{pqrmatrix} with $q = r$ where $M_{pq} = M_{qp}^t$ and $M_{qq}$ is symmetric;
  the dimension is $pq + \tfrac12 q(q{+}1)$.
  \end{itemize}
We consider the smallest non-trivial system $J$ in each family: $A_2$, $B_2$, $C_{111}$, $D_{11}$ with dimensions
4, 3, 3, 2 respectively.
To construct the universal associative envelope $U(J)$,
we use noncommutative Gr\"obner bases in free associative algebras \cite{BremnerGrobner}:
$U(J)$ is the quotient $F\langle B \rangle / I(G)$ of the free associative algebra $F\langle B \rangle$
on a basis $B$ of $J$ by the ideal $I(G)$ generated by the following set of relations:
  \[
  G = \big\{ \, abcd + dcba - \{a{,}b{,}c{,}d\} \mid a, b, c, d \in B \, \big\}.
  \]
We find a Gr\"obner basis of $I(G)$ and identify the monomials in $F\langle B \rangle$
which do not have the leading monomial of any Gr\"obner basis element as a subword.
The cosets of these monomials form a basis for $U(J)$.
If $U(J)$ is finite dimensional, then we can determine its structure using
the Wedderburn decomposition of associative algebras \cite{BremnerWedderburn}.
We consider the four smallest systems by increasing dimension.

\begin{proposition} \label{UD11}
We have $U(D_{11}) \cong \mathbb{Q} \oplus M_3(\mathbb{Q})$,
and hence up to isomorphism, $D_{11}$ has only two finite dimensional irreducible representations,
the 1-dimensional trivial representation and the 3-dimensional natural representation.
\end{proposition}

\begin{proof}
In the natural representation by $3 \times 3$ matrices, $D_{11}$ has this basis:
  \[
  a = \begin{bmatrix} 0 & 0 & 1 \\ 1 & 0 & 0 \\ 0 & 0 & 0 \end{bmatrix},
  \qquad
  b = \begin{bmatrix} 0 & 0 & 0 \\ 0 & 0 & 0 \\ 0 & 1 & 0 \end{bmatrix}.
  \]
The only nonzero quadruple products are $\{a,b,a,a\} = a$ and $\{b,a,a,b\} = b$.
Hence $U(D_{11}) = F\langle a, b \rangle / I(G)$ where $I(G)$ is generated by the following self-reduced set
$\{ g_1, \dots, g_{10} \}$ in lex order of leading monomials with $a \prec b$:
  \begin{align*}
  &
  a^4, \quad
  aba^2 + a^2ba - a, \quad
  ab^2a, \quad
  ba^3 + a^3b, \quad
  ba^2b - b,
  \\
  &
  baba + abab, \quad
  b^2a^2 + a^2b^2, \quad
  b^2ab + bab^2, \quad
  b^3a + ab^3, \quad
  b^4.
  \end{align*}
We make the Gr\"obner basis algorithm deterministic by considering the pairs $( g_i, g_j )$ in lex order
of the subscripts $(i,j)$;
whenever their composition is reducible, we reduce the highest term in deglex order using the generator
whose leading monomial is lowest in deglex order.
For example, $(g_2,g_1)$ produces this composition:
  \begin{align*}
  g_2 a^2 - ab g_1
  &=
  ( aba^2 + a^2ba - a ) a^2 - ab a^4
  =
  aba^4 + a^2ba^3 - a^3 - aba^4
  =
  a^2ba^3 - a^3.
  \end{align*}
We reduce this composition in three steps using $g_2, g_2, g_1$ as follows:
  \begin{align*}
  a^2ba^3 - a^3 - a g_2 a
  &=
  a^2ba^3 - a^3 - a ( aba^2 + a^2ba - a ) a
  =
  - a^3ba^2,
  \\
  - a^3ba^2 + a^2 g_2
  &=
  - a^3ba^2 + a^2 ( aba^2 + a^2ba - a )
  =
  a^4ba - a^3,
  \\
  a^4ba - a^3 - g_1 ba
  &=
  a^4ba - a^3 - a^4 ba = - a^3.
  \end{align*}
We cannot reduce further, so we record $a^3$ as a (monic) composition.
This process gives the following set $\{ h_1, \dots, h_{24} \}$ of distinct nonzero reduced compositions:
  \begin{align*}
  &
  a^3, \quad
  ab^2, \quad
  bab, \quad
  b^2a, \quad
  b^3, \quad
  a^3b, \quad
  a^2b^2, \quad
  abab, \quad
  ab^3, \quad
  bab^2, \quad
  a^3ba,
  \\
  &
  a^3b^2, \quad
  a^2bab, \quad
  a^2b^3 + b^2, \quad
  abab^2, \quad
  bab^3, \quad
  a^3bab, \quad
  a^2bab^2 - a^3b^3 - ab^2,
  \\
  &
  a^2bab^2 - bab, \quad
  a^2bab^2, \quad
  abab^3, \quad
  a^3bab^2 -2 a^2b^2, \quad
  a^2bab^3, \quad
  a^2bab^3 - ab^3.
  \end{align*}
We sort $\{ g_1, \dots, g_{10}, h_1, \dots, h_{24} \}$ by deglex order of leading monomials,
and compute the normal form of each element with respect to the preceding elements.
Many normal forms are 0, showing that $I(G)$ is generated by these five elements:
  \[
  b^2, \qquad
  a^3, \qquad
  bab, \qquad
  aba^2 + a^2ba - a, \qquad
  ba^2b - b.
  \]
All compositions of these elements reduce to 0, so we have a Gr\"obner basis for $I(G)$.
Only $d = 10$ monomials in $F\langle a, b \rangle$ are not divisible any of the leading monomials:
  \[
  1, \quad a, \quad b, \quad a^2, \quad ab, \quad ba, \quad a^2b, \quad aba, \quad ba^2, \quad a^2ba.
  \]
The cosets of these monomials modulo $I(G)$ form a basis for $U(D_{11})$.
We calculate the structure constants for $U(D_{11})$ by computing the normal forms of the products
of these monomials with respect to the Gr\"obner basis; see Table \ref{D111table}.

\begin{table}[h]
\[
\begin{array}{c|cccccccccc}
& 1 &   a &   b &   a^2 &   ab &   ba &   a^2b &   aba &   ba^2 &   a^2ba \\ \midrule
1 & 1 & a & b & a^2 & ab & ba & a^2b & aba & ba^2 & a^2ba \\
a & a & a^2 & ab &  0  & a^2b & aba &  0  & a^2ba & -a^2ba{+}a &  0  \\
b & b & ba &  0  & ba^2 &  0  &  0  & b &  0  &  0  & ba \\
a^2 & a^2 &  0  & a^2b &  0  &  0  & a^2ba &  0  &  0  & a^2 &  0  \\
ab & ab & aba &  0  & -a^2ba{+}a &  0  &  0  & ab &  0  &  0  & aba \\
ba & ba & ba^2 &  0  &  0  & b &  0  &  0  & ba &  0  &  0  \\
a^2b & a^2b & a^2ba &  0  & a^2 &  0  &  0  & a^2b &  0  &  0  & a^2ba \\
aba & aba & -a^2ba{+}a &  0  &  0  & ab &  0  &  0  & aba &  0  &  0  \\
ba^2 & ba^2 &  0  & b &  0  &  0  & ba &  0  &  0  & ba^2 &  0  \\
a^2ba & a^2ba & a^2 &  0  &  0  & a^2b &  0  &  0  & a^2ba &  0  &  0  \\
\midrule
\end{array}
\]
\caption{Structure constants for $U(D_{11})$}
\label{D111table}
\end{table}

The radical of $U(D_{11})$ is the nullspace of the Dickson matrix \cite[Corollary 12]{BremnerGrobner}:
the $d \times d$ matrix $\Delta$ defined in terms of the structure constants $c_{ij}^k$ for $U(D_{11})$:
  \[
  \qquad
  \Delta
  =
  \footnotesize
  \begin{bmatrix}
  10 & \cdot & \cdot & \cdot & \cdot & \cdot & 3 & 3 & 3 & \cdot \\
   \cdot & \cdot & \cdot & \cdot & 3 & 3 & \cdot & \cdot & \cdot & \cdot \\
   \cdot & \cdot & \cdot & 3 & \cdot & \cdot & \cdot & \cdot & \cdot & \cdot \\
   \cdot & \cdot & 3 & \cdot & \cdot & \cdot & \cdot & \cdot & \cdot & \cdot \\
   \cdot & 3 & \cdot & \cdot & \cdot & \cdot & \cdot & \cdot & \cdot & 3 \\
   \cdot & 3 & \cdot & \cdot & \cdot & \cdot & \cdot & \cdot & \cdot & \cdot \\
   3 & \cdot & \cdot & \cdot & \cdot & \cdot & 3 & \cdot & \cdot & \cdot \\
   3 & \cdot & \cdot & \cdot & \cdot & \cdot & \cdot & 3 & \cdot & \cdot \\
   3 & \cdot & \cdot & \cdot & \cdot & \cdot & \cdot & \cdot & 3 & \cdot \\
   \cdot & \cdot & \cdot & \cdot & 3 & \cdot & \cdot & \cdot & \cdot & \cdot
  \end{bmatrix}
  \qquad\qquad
  \left(
  \Delta_{ij}
  =
  \sum_{k=1}^d \sum_{\ell=1}^d
  c_{ji}^k c_{k\ell}^\ell
  \right)
  \]
This matrix has full rank, so $U(D_{11})$  is semisimple.
The center $Z(U(D_{11}))$ is the nullspace of the $d^2 \times d$ matrix from \cite[Corollary 15]{BremnerGrobner};
the RCF of this matrix is
  \[
  \footnotesize
  \begin{bmatrix}
  \cdot & 1 & \cdot & \cdot & \cdot & \cdot & \cdot & \cdot & \cdot & \cdot  \\
  \cdot & \cdot & 1 & \cdot & \cdot & \cdot & \cdot & \cdot & \cdot & \cdot  \\
  \cdot & \cdot & \cdot & 1 & \cdot & \cdot & \cdot & \cdot & \cdot & \cdot  \\
  \cdot & \cdot & \cdot & \cdot & 1 & \cdot & \cdot & \cdot & \cdot & \cdot  \\
  \cdot & \cdot & \cdot & \cdot & \cdot & 1 & \cdot & \cdot & \cdot & \cdot  \\
  \cdot & \cdot & \cdot & \cdot & \cdot & \cdot & 1 & \cdot & -1 & \cdot  \\
  \cdot & \cdot & \cdot & \cdot & \cdot & \cdot & \cdot & 1 & -1 & \cdot  \\
  \cdot & \cdot & \cdot & \cdot & \cdot & \cdot & \cdot & \cdot & \cdot & 1
  \end{bmatrix}
  \]
Hence a basis for $Z(U(D_{11}))$ consists of (the cosets of) 1 and $c = a^2b + aba + ba^2$.
We calculate $c^2 = c$, so $x(x-1)$ is the minimal polynomial of $c$.
Hence a basis for $Z(U(D_{11}))$ consisting of orthogonal primitive idempotents
is $\{ c, 1-c \}$.

We compute the two-sided ideals in $U(D_{11})$ generated by $c$ and $1-c$.
We find that $\dim \langle 1-c \rangle = 1$ so that $\langle 1-c \rangle \cong \mathbb{F}$ with basis
$\mathbbm{1} = 1 - a^2b - aba - ba^2$.
Moreover, $\dim \langle c \rangle = 9$, so we expect that $\langle c \rangle \cong M_3(\mathbb{F})$.
To obtain an isomorphism, we compute a basis for $\langle c \rangle$, and find that it consists
of the leading monomials of the Gr\"obner basis (excluding 1).
The coset of $b$ generates a 3-dimensional left ideal in $\langle c \rangle$ with
basis $e_1 = b$, $e_2 = ab$, $e_3 = a^2b$.
We identify these with the standard basis vectors in $\mathbb{F}^3$
and determine the elements of $\langle c \rangle$ corresponding to the matrix units $E_{ij}$
($1 \le i, j \le 3$).
For each $i, j$ we determine $c_1, \dots, c_9 \in \field$ so that the coset of the generic element,
  \[
  E = c_1 a + c_2 b + c_3 a^2 + c_4 ab + c_5 ba + c_6 a^2b + c_7 aba + c_8 ba^2 + c_9 a^2ba,
  \]
acts like the matrix unit $E_{ij}$; that is, $E_{ij} e_k = \delta_{jk} e_i$.
From Table \ref{D111table} we obtain:
  \begin{align*}
  E b
  &=
  c_8 b + c_1 ab + c_3 a^2b,
  \quad
  E ab
  =
  c_5 b + c_7 ab + ( c_1 {+} c_9 ) a^2b,
  \quad
  E a^2 b
  =
  c_2 b + c_4 ab + c_6 a^2b.
  \end{align*}
From this we obtain linear systems which determine $c_1, \dots, c_9$ for each $E_{ij}$,
and solving these systems gives this isomorphism $\langle c \rangle \cong M_3(\mathbb{F})$:
  \[
  \big[
  E_{11}, E_{12}, E_{13}, E_{21}, E_{22}, E_{23}, E_{31}, E_{32}, E_{33}
  \big]
  \longleftrightarrow
  \big[
  ba^2, ba, b, a{-}a^2ba, aba, ab, a^2, a^2ba, a^2b
  \big].
  \]
The set $\{ \, \mathbbm{1}, \, E_{ij} \mid 1 \le i, j \le 3 \, \}$ is a basis for $U(D_{11})$ which reflects
the decomposition into simple two-sided ideals, and proves the proposition.
\end{proof}

\begin{openproblem}
Determine the structure of $U(D_{pq})$ for all $p, q$.
In particular, do we always have the isomorphism $U(D_{pq}) \cong \field \oplus M_n(\field)$ where
$n = pq + \tfrac12 q(q{+}1)$?
\end{openproblem}

\begin{proposition} \label{propUC111}
We have $U(C_{111}) \cong \mathbb{Q} \oplus 2 \, M_3(\mathbb{Q})$,
and hence up to isomorphism, $C_{111}$ has only three finite dimensional irreducible representations:
1-dimensional trivial, 3-dimensional natural, and another 3-dimensional.
\end{proposition}

\begin{proof}
Our basis for $C_{111}$ consists of matrices defining the natural representation:
  \[
  a = \begin{bmatrix} 0 & 0 & 0 \\ 1 & 0 & 0 \\ 0 & 0 & 0 \end{bmatrix}, \qquad
  b = \begin{bmatrix} 0 & 0 & 0 \\ 0 & 0 & 0 \\ 0 & 1 & 0 \end{bmatrix}, \qquad
  c = \begin{bmatrix} 0 & 0 & 1 \\ 0 & 0 & 0 \\ 0 & 0 & 0 \end{bmatrix}.
  \]
The nonzero quadrilinear products are:
  $\{ a, c, b, a \} = a$, 
  $\{ b, c, a, b \} = b$, 
  $\{ c, b, a, c \} = c$.
The original generating set for $I(G)$ contains 45 elements and is already self-reduced.
We obtain 290 distinct nontrivial compositions;
the resulting set of 335 generators collapses to a self-reduced set of only 13 elements, which is a Gr\"obner basis:
  \begin{align*}
  &
  a^2, \quad
  b^2, \quad
  c^2, \quad
  aba, \quad
  aca, \quad
  bab, \quad
  bcb, \quad
  cac, \quad
  cbc,
  \\
  &
  acba + abca - a, \quad
  bcab + bacb - b, \quad
  cbac + cabc - c, \quad
  cabca - ca.
  \end{align*}
Only 19 monomials in $F\langle a, b, c \rangle$ do not have a leading monomial of a Gr\"obner basis element as a subword:
  \[
  1, \;
  a, \;
  b, \;
  c, \;
  ab, \;
  ac, \;
  ba, \;
  bc, \;
  ca, \;
  cb, \;
  abc, \;
  acb, \;
  bac, \;
  bca, \;
  cab, \;
  cba, \;
  abca, \;
  bacb, \;
  cabc.
  \]
Hence $U( C_{111} )$ has dimension 19; as in the proof of Proposition \ref{UD11} we find that $U( C_{111} )$ is semisimple,
and its center has dimension 3.
The cosets of these elements form a basis of the center consisting of orthogonal primitive idempotents:
  \[
  1 - abc - acb - bac - bca - cab - cba, \qquad
  abc  + bca + cab, \qquad
  acb + bac + cba.
  \]
These elements generate simple two-sided ideals of dimensions 1, 9, 9 respectively.
We omit the isomorphisms of each simple ideal with a full matrix algebra.
\end{proof}

\begin{openproblem}
Determine the structure of the universal associative envelope of $C_{pqr}$ for all $p, q, r$.
Do we always have the isomorphism $U( C_{pqr} ) \cong \field \oplus 2 \, M_n(\field)$ where $n = pq + qr + rp$?
\end{openproblem}

\begin{proposition}
We have $U(B_2) \otimes_\field \mathbb{K} \cong \mathbb{K} \oplus 3 \, M_2(\mathbb{K})$
where $\mathbb{K} = \field(\beta)$ with $\beta = 1+\sqrt{-3}$.
Hence over this quadratic extension field, $B_2$ has only four finite dimensional irreducible representations
up to isomorphism: 1-dimensional trivial, 2-dimensional natural,
and another two 2-dimensional.
\end{proposition}

\begin{proof}
Our basis for $B_2$ consists of matrices defining the natural representation:
  \[
  a = \begin{bmatrix} 1 & 0 \\ 0 & 0 \end{bmatrix}, \qquad
  b = \begin{bmatrix} 0 & 0 \\ 0 & 1 \end{bmatrix}, \qquad
  c = \begin{bmatrix} 0 & 1 \\ 1 & 0 \end{bmatrix}.
  \]
There are 19 nonzero quadrilinear products of these generators (which we omit).
The original generating set for $I(G)$ contains 45 elements and is already self-reduced.
We obtain 533 distinct nontrivial compositions;
the resulting set of 578 generators collapses to a self-reduced set of only 8 elements, which is a Gr\"obner basis:
  \[
  ab, \quad
  ba, \quad
  ca - bc, \quad
  cb - ac, \quad
  c^2 - b^2 - a^2, \quad
  a^4 - a, \quad
  b^4 - b, \quad
  b^3c + a^3c - c.
  \]
Only 13 monomials in $F\langle a, b, c \rangle$ do not have a leading monomial of a Gr\"obner basis element as a subword:
  \[
  1, \quad
  a, \quad
  b, \quad
  c, \quad
  a^2, \quad
  ac, \quad
  b^2, \quad
  bc, \quad
  a^3, \quad
  a^2c, \quad
  b^3, \quad
  b^2c, \quad
  a^3c.
  \]
Hence $U( B_2 )$ has dimension 13; as before, we find that $U( B_2 )$ is semisimple.
The center has dimension 4 and the following basis:
1, $a+b$, $a^2 + b^2$, $a^3 + b^3$.
These elements are not orthogonal primitive idempotents;
unlike the previous two cases, to split the center we must extend the base field.
\begin{itemize}
\item
First splitting: set $e = 1$ and $x = a+b$.
The minimal polynomial of $x$ is $t^4 - t = t(t^3-1)$, so we set $y = x$ and $z = x^3-e$.
Then $y$ generates a 3-dimensional central ideal with identity element $a^3 + b^3$, and
$\zeta_1 = -z$ is a primitive idempotent.
\item
Second splitting: set $e = a^3 + b^3$ and $x = a + b$.
The minimal polynomial of $x$ is $t^3 - 1 = (t-1)(t^2+t+1)$, so we set $y = x-e$ and $z = x^2+x+e$.
Then $y$ generates a 2-dimensional central ideal with identity element
$-\frac13(a+b) - \frac13(a^2+b^2) + \frac23(a^3 + b^3)$,
and $\zeta_2 = \frac13 z$ is a primitive idempotent.
\item
Third splitting: set $e = -\frac13(a+b) - \frac13(a^2+b^2) + \frac23(a^3 + b^3)$ and $x = (a + b) - (a^3 + b^3)$.
The minimal polynomial of $x$ is $t^2 + 3t + 3 = (t - \alpha)(t - \bar{\alpha})$ where $\alpha = \frac12( -3 + \sqrt{-3} )$.
We set $y = x-\alpha e$ and $z = x-\bar\alpha e$; each generates a 1-dimensional central ideal.
We get the idempotents $\zeta_3 = \frac{-1}{\sqrt{-3}} y$ and $\zeta_4 = \frac{1}{\sqrt{-3}} z$.
\end{itemize}
Using the fact that $ab = ba = 0$ in $U( B_2 )$ since these monomials are in the Gr\"obner basis,
we obtain the central basis of orthogonal primitive idempotents:
  \begin{align*}
  \zeta_1 &= 1 - a^3 - b^3,
  \qquad
  \zeta_2 = \tfrac13 ( a + b + a^2 + b^2 + a^3 + b^3 ),
  \\
  \zeta_3 &= -\tfrac16 \big( \beta(a+b) + \bar\beta(a^2+b^2) - 2(a^3+b^3) \big),
  \\
  \zeta_4 &= -\tfrac16 \big( \bar\beta(a+b) + \beta(a^2+b^2) - 2(a^3+b^3) \big),
  \end{align*}
where $\beta = 1+\sqrt{-3}$.
These four elements generate simple two-sided ideals of dimensions 1, 4, 4, 4 respectively.
We omit the calculation of the isomorphisms of each simple ideal with a full matrix algebra.
\end{proof}

\begin{openproblem}
Determine the structure of the universal associative envelope of $B_n$ for all $n$.
In particular, after a quadratic extension $\mathbb{K}$ of $\field$,
do we always have the isomorphism $U( B_n ) \otimes_\field \mathbb{K} \cong \mathbb{K} \oplus 3 \, M_n(\mathbb{K})$?
\end{openproblem}

\begin{proposition} \label{propUA2}
We have $U(A_2) \otimes_\field \mathbb{K} \cong \mathbb{K} \oplus 6 \, M_2(\mathbb{K})$
where $\mathbb{K} = \field(\beta)$, $\beta = 1+\sqrt{-3}$.
Hence over this quadratic extension, $A_2$ has exactly seven finite dimensional irreducible representations
up to isomorphism:
1-dimensional trivial, 2-dimensional natural,
and another five 2-dimensional.
\end{proposition}

\begin{proof}
Our basis for $A_2$ consists of these matrices in the natural representation:
  \[
  a = \begin{bmatrix} 1 & 0 \\ 0 & 0 \end{bmatrix}, \qquad
  b = \begin{bmatrix} 0 & 1 \\ 0 & 0 \end{bmatrix}, \qquad
  c = \begin{bmatrix} 0 & 0 \\ 1 & 0 \end{bmatrix}, \qquad
  d = \begin{bmatrix} 0 & 0 \\ 0 & 1 \end{bmatrix}.
  \]
There are 31 nonzero quadrilinear products of these generators (which we omit).
The original generating set of $I(G)$ contains 136 elements and is already self-reduced.
We obtain 2769 distinct nontrivial compositions;
the resulting 2905 generators collapses to a self-reduced set of 25 elements, which is a Gr\"obner basis:
  \begin{align*}
  &
  ad, \quad
  b^2, \quad
  bd - ab, \quad
  c^2, \quad
  cd - ac, \quad
  da, \quad
  db - ba, \quad
  dc - ca, \quad
  d^2 - cb - bc + a^2,
  \\
  &
  aba, \quad
  aca, \quad
  acb + abc - a^3, \quad
  bab, \quad
  bca - abc, \quad
  bcb - ba^2 - a^2b, \quad
  cac,
  \\
  &
  cba + abc - a^3, \quad
  cbc - ca^2 - a^2c, \quad
  a^4 - a, \quad
  ba^3 + a^3b - b, \quad
  ba^2b, \quad
  ca^3 + a^3c - c,
  \\
  &
  ca^2b + ba^2c - d, \quad
  ca^2c, \quad
  cabc + a^3c - c.
  \end{align*}
Only 25 monomials in $F\langle a, b, c, d \rangle$ do not have a leading monomial of a Gr\"obner basis element as a subword:
  \begin{align*}
  &
  1, \quad
  a, \quad
  b, \quad
  c, \quad
  d, \quad
  a^2, \quad
  ab, \quad
  ac, \quad
  ba, \quad
  bc, \quad
  ca, \quad
  cb, \quad
  a^3, \quad
  a^2b,
  \\
  &
  a^2c, \quad
  abc, \quad
  ba^2, \quad
  bac, \quad
  ca^2, \quad
  cab, \quad
  a^3b, \quad
  a^3c, \quad
  a^2bc, \quad
  ba^2c, \quad
  a^3bc.
  \end{align*}
Hence $U( A_2 )$ has dimension 25; as before, we find that $U( A_2 )$ is semisimple.
The center has dimension 7 and the following basis:
  \[
  1, \quad
  a+d, \quad
  bc+cb, \quad
  a^3-abc+bac, \quad
  abc+cab, \quad
  a-a^2bc+ba^2c, \quad
  -\tfrac12(a^2+bc)+a^3bc.
  \]
These elements are not orthogonal primitive idempotents;
as in the previous case, to split the center we must extend the base field.
We obtain the following central basis consisting of orthogonal primitive idempotents:
  \begin{align*}
  &\!\!\!\!
  1 - a^3 - bac - cab,
  \\
  \tfrac13
  &\big(
  d - a^2 + cb + abc + cab + a^2bc - ba^2c + 2 a^3bc
  \big),
  \\
  \tfrac13
  &\big(
  a + a^2 + bc + a^3 - abc + bac - a^2bc + ba^2c - 2 a^3bc
  \big),
  \\
  -\tfrac16
  &\big(
  \beta  a
  + \bar\beta  a^2
  + \bar\beta  bc
  - 2 a^3
  + 2 abc
  - 2 bac
  - \beta  a^2bc
  + \beta  ba^2c
  - 2 \bar\beta  a^3bc
  \big),
  \\
  -\tfrac16
  &\big(
  \beta  d
  - \bar\beta  a^2
  + \bar\beta  cb
  - 2  abc
  - 2  cab
  + \beta  a^2bc
  - \beta  ba^2c
  + 2\bar\beta  a^3bc
  \big),
  \\
  -\tfrac16
  &\big(
  \bar\beta  a
  + \beta  a^2
  + \beta  bc
  - 2  a^3
  + 2  abc
  - 2  bac
  - \bar\beta  a^2bc
  + \bar\beta  ba^2c
  - 2 \beta  a^3bc
  \big),
  \\
  -\tfrac16
  &\big(
  \bar\beta  d
  - \beta  a^2
  + \beta  cb
  -2  abc
  -2  cab
  + \bar\beta  a^2bc
  - \bar\beta  ba^2c
  + 2\beta  a^3bc
  \big),
  \end{align*}
where $\beta = 1 + \sqrt{-3}$.
These seven elements generate simple two-sided ideals of dimensions 1, 4, 4, 4, 4, 4, 4 respectively.
We omit the calculation of the isomorphisms of each simple ideal with a full matrix algebra.
\end{proof}

\begin{openproblem}
Determine the structure of the universal associative envelope for $A_n$ for all $n$.
In particular, after making a quadratic extension $\mathbb{K}$ of $\field$,
do we always have the isomorphism $U( A_n ) \otimes_\field \mathbb{K} \cong \mathbb{K} \oplus 6 \, M_n(\mathbb{K})$?
\end{openproblem}


\section{Defining identities for anti-Jordan quadruple systems} \label{sectionAJQS}

In this and the next two sections we present analogous results for the anti-tetrad
$[a,b,c,d] = abcd - dcba$.
Since the methods are very similar, we omit most details.

\begin{lemma} \label{antitetrad4}
Every multilinear identity in degree 4 satisfied by the anti-tetrad is a consequence of the skew-symmetry
$[a,b,c,d] + [d,c,b,a] \equiv 0$.
\end{lemma}

\begin{theorem} \label{antitetrad7}
Every multilinear identity in degree 7 satisfied by the anti-tetrad is a consequence of
skew-symmetry in degree 4 and these two identities in degree 7:
  \begin{align*}
  & [[a,b,c,d],e,f,g]  
  - [[a,b,f,e],d,c,g]  
  + [[d,c,f,e],a,b,g]  
  + [g,[b,a,d,c],f,e]  
  \\
  &
  {}
  - [g,[b,a,e,f],c,d]  
  + [g,[c,d,e,f],b,a]  
  \equiv 0,
  \\[4pt]
  & [[a,b,c,d],e,f,g]  
  - [[a,f,g,c],b,e,d]  
  + [[c,a,b,d],g,e,f]  
  - [[c,b,a,f],g,e,d]  
  \\
  &
  {}
  - [[c,g,d,b],a,e,f]  
  - [[c,g,e,f],a,b,d]  
  + [[c,g,e,d],b,a,f]  
  - [[b,a,c,f],e,d,g]  
  \\
  &
  {}
  - [[b,c,f,g],a,e,d]  
  + [[b,e,d,g],a,c,f]  
  - [[g,f,e,d],c,b,a]  
  + [[g,d,e,f],c,a,b]  
  \\
  &
  {}
  - [[f,a,b,d],e,g,c]  
  - [[f,c,e,d],g,a,b]  
  - [a,[b,c,g,f],e,d]  
  - [c,[a,b,d,g],e,f]  
  \\
  &
  {}
  + [c,[b,a,f,g],e,d]  
  + [b,[a,c,g,d],e,f]  
  - [b,[a,g,f,c],e,d]  
  + [b,[a,g,d,e],c,f]  
  \\
  &
  {}
  + [d,[c,b,a,e],f,g]  
  \equiv 0.
  \end{align*}
\end{theorem}

\begin{proof}
For the anti-tetrad the rank is 2519 and the nullity is 2521.
\end{proof}

\begin{definition} \label{defajqs}
An \textbf{anti-Jordan quadruple system} (\textbf{AJQS}) is a vector space $Q$
over a field $\mathbb{F}$ of characteristic 0 or $p > 7$
with a quadrilinear map $[-,-,-,-]\colon Q^4 \to Q$ satisfying the skew-symmetry of
Lemma \ref{antitetrad4} and the identities of Theorem \ref{antitetrad7}.
\end{definition}

\begin{table}[h]
\begin{tabular}{lr|rrr|rrrr|r}
&\; & \multicolumn{3}{|c|}{symmetries} & \multicolumn{4}{|c|}{expansions}
\\
\midrule
$\lambda$ &\; $d_\lambda$ &\; rows &\; cols &\; sym &\; rows &\; cols &\; rank &\; null &\; new
\\
\midrule
  $7$      &   1 &   3 &   2 &   2 &   2 &   1 &   0 &   2 &   0 \\
  $61$     &   6 &  18 &  12 &   8 &  12 &   6 &   3 &   9 &   1 \\
  $52$     &  14 &  42 &  28 &  16 &  28 &  14 &   8 &  20 &   4 \\
  $51^2$   &  15 &  45 &  30 &  14 &  30 &  15 &   6 &  24 &  10 \\
  $43$     &  14 &  42 &  28 &  16 &  28 &  14 &   7 &  21 &   5 \\
  $421$    &  35 & 105 &  70 &  34 &  70 &  35 &  18 &  52 &  18 \\
  $41^3$   &  20 &  60 &  40 &  16 &  40 &  20 &  10 &  30 &  14 \\
  $3^21$   &  21 &  63 &  42 &  22 &  42 &  21 &   9 &  33 &  11 \\
  $32^2$   &  21 &  63 &  42 &  22 &  42 &  21 &  12 &  30 &   8 \\
  $321^2$  &  35 & 105 &  70 &  34 &  70 &  35 &  17 &  53 &  19 \\
  $31^4$   &  15 &  45 &  30 &  14 &  30 &  15 &   9 &  21 &   7 \\
  $2^31$   &  14 &  42 &  28 &  16 &  28 &  14 &   7 &  21 &   5 \\
  $2^21^3$ &  14 &  42 &  28 &  16 &  28 &  14 &   6 &  22 &   6 \\
  $21^5$   &   6 &  18 &  12 &   8 &  12 &   6 &   3 &   9 &   1 \\
  $1^7$    &   1 &   3 &   2 &   2 &   2 &   1 &   0 &   2 &   0 \\
\midrule
\end{tabular}
\medskip
\caption{$S_7$-module multiplicities for the anti-tetrad in degree 7}
\label{antitetrad7table}
\end{table}

\begin{lemma} \label{antitetrad7lemma}
Over a field of characteristic 0 or $p > 7$, for each partition $\lambda$
column ``new'' of Table \ref{antitetrad7table} gives the multiplicity
of $[\lambda]$ in the kernel of the expansion map for the anti-tetrad.
\end{lemma}

Since the multiplicity for the last partition is 0, we expect to find simpler nonlinear identities
whose linearizations generate the $S_7$-module of all multilinear identities for the anti-tetrad
in degree 7 (compare Remark \ref{tetradnonlinear}).
This computation consists of two parts.

\subsubsection*{Part 1}

For each partition of 7 except the last, we construct the corresponding multiset of variables:
  \begin{align*}
  &
  a^7, \quad
  a^6 b, \quad
  a^5 b^2, \quad
  a^5 b c, \quad
  a^4 b^3, \quad
  a^4 b^2 c, \quad
  a^4 b c d, \quad
  a^3 b^3 c,
  \\
  &
  a^3 b^2 c^2, \quad
  a^3 b^2 c d, \quad
  a^3 b c d e, \quad
  a^2 b^2 c^2 d, \quad
  a^2 b^2 c d e, \quad
  a^2 b c d e f.
  \end{align*}
For each partition $\lambda$, we generate all permutations of the multiset;
we write $A_\lambda$ for the number of these associative monomials.
We then apply each of the two association types to all the permutations,
and straighten them using the skew-symmetry in degree 4, to obtain a set of monomials in normal form;
we write $Q_\lambda$ for the number of these quaternary monomials.
For each $\lambda$, the nonlinear expansion matrix $E_\lambda$ has size $A_\lambda \times Q_\lambda$.
We write $R_\lambda$ and $N_\lambda$ for the rank and nullity of the expansion matrix.
We use HNF and LLL to compute a lattice basis for the integer nullspace of $E_\lambda$
and sort these vectors by increasing Euclidean length.
We linearize the nonlinear identities represented by the nullspace basis vectors,
and then study the $S_7$-module generated by these multilinear identities.
We extract a smaller subset of these identities which generates the $S_7$-module, as follows.
For each partition $\lambda$ we retain only those identities which do not belong to the submodule generated
by all previous identities (for $\lambda$ and all previous partitions).
We write $G_\lambda$ for the number of generators obtained in this way for partition $\lambda$.
We summarize the results of these computations:
  \[
  \begin{array}{l|rrrrrrrrrrrrrr}
  \lambda &
  7 & 6 1 & 5 2 & 5 1^2 & 4 3 & 4 2 1 & 4 1^3 & 3^2 1 & 3 2^2 & 3 2 1^2 & 3 1^4 & 2^3 1 & 2^2 1^3 & 2 1^5
  \\
  \midrule
  A_\lambda &
  1 & 7 & 21 & 42 & 35 & 105 & 210 & 140 & 210 & 420 & 840 & 630 & 1260 & 2520
  \\
  Q_\lambda &
  0 & 4 & 16 & 36 & 28 & 96 & 204 & 128 & 196 & 408 & 840 & 612 & 1248 & 2520
  \\
  R_\lambda &
  0 & 3 & 11 & 20 & 18 & 53 & 104 & 69 & 107 & 209 & 419 & 317 & 629 & 1259
  \\
  N_\lambda &
  0 & 1 & 5 & 16 & 10 & 43 & 100 & 59 & 89 & 199 & 421 & 295 & 619 & 1261
  \\
  G_\lambda &
  0 & 1 & 4 & 10 & 5 & 18 & 14 & 11 & 8 & 19 & 7 & 5 & 6 & 1
  \end{array}
  \]
Altogether this gives 109 nonlinear identities whose linearizations generate the
2521-dimensional $S_7$-module of all multilinear identities satisfied by the anti-tetrad in degree 7
(modulo the consequences of the skew-symmetry in degree 4).

\subsubsection*{Part 2}

Starting at the end of the list of module generators obtained from Part 1, and working toward the beginning of the list,
we remove the current (linearized) identity from the list and compute the dimension of the submodule generated by
the remaining identities.
If the dimension reaches 2521, then the current identity is redundant, and we remove it from the list.
If this happens, the current identity is a consequence of identities which follow it in the list;
by the way in which we extracted these generators in Part 1, we know that no identity is a consequence of the identities
which precede it in the list.
The result of this elimination process is a set of only three nonlinear identities with respectively 5, 5 and 12 terms,
whose coefficients are all $\pm 1$.
Even though we have three identities instead of the two identities of Theorem \ref{antitetrad7} with 6 and 21 terms,
the total number of terms is smaller.
We summarize the results of these computations in the next result.

\begin{theorem}
Every multilinear identity in degree 7 satisfied by the anti-tetrad is a consequence of
the skew-symmetry in degree 4 and the linearizations of these three nonlinear identities in degree 7,
corresponding to partitions $2^2 1^3$ (twice) and $2 1^5$:
  \begin{align*}
  &
  [ [ a, b, c, b ], d, e, a ]    
  - [ [ a, b, c, b ], e, d, a ]  
  - [ [ a, d, e, a ], b, c, b ]  
  - [ a, [ b, c, b, d ], e, a ]  
  \\
  &
  {}
  + [ a, [ b, c, b, e ], d, a ]  
  \equiv 0,
  \\[4pt]
  &
  [ [ a, c, a, b ], b, d, e ]    
  - [ [ a, c, a, e ], d, b, b ]  
  + [ [ b, b, d, e ], a, c, a ]  
  + [ b, [ a, c, a, b ], d, e ]  
  \\
  &
  {}
  - [ e, [ a, c, a, d ], b, b ]  
  \equiv 0,
  \\[4pt]
  &
  [ [ a, a, b, c ], d, e, f ]    
  + [ [ a, a, e, f ], d, c, b ]  
  - [ [ a, a, e, b ], c, d, f ]  
  - [ [ a, d, c, b ], a, e, f ]  
  \\
  &
  {}
  - [ [ f, a, a, c ], d, e, b ]  
  + [ [ f, d, c, b ], e, a, a ]  
  + [ [ f, d, e, b ], a, a, c ]  
  - [ a, [ a, b, c, d ], e, f ]  
  \\
  &
  {}
  - [ f, [ a, a, d, e ], b, c ]  
  + [ f, [ a, a, c, d ], e, b ]  
  + [ f, [ a, b, e, d ], a, c ]  
  - [ c, [ a, d, e, b ], a, f ]  
  \equiv 0.
  \end{align*}
\end{theorem}


\section{Special identities in degree 10 for the anti-tetrad}
\label{sectionspecialantitetrad}

The notions of special and exceptional for AJQS, and special identities for the anti-tetrad,
are the obvious analogues of Definition \ref{defjqsspecial}.

\begin{openproblem}
Do there exist exceptional AJQS?
\end{openproblem}

\begin{theorem} \label{antitetrad10}
For the anti-tetrad, and for each partition $\lambda$ of 10, columns ``symm'', ``symmlift'', ``null'' and ``new''
in the corresponding row of Table \ref{antitetrad10table} contain respectively the multiplicity of $[\lambda]$
in the modules $\mathrm{Symm}(10)$, $\mathrm{Old}(10)$, $\mathrm{All}(10)$ and $\mathrm{New}(10)$.
\end{theorem}

\begin{openproblem}
Determine a minimal set of multilinear identities in degree 10 for the anti-tetrad which generate $\mathrm{New}(10)$
as an $S_{10}$-module.
\end{openproblem}

\begin{table} \small
\[
\begin{array}{r|lr|rrrr|rrrr|r}
&\; &\; & \multicolumn{4}{|c|}{\text{symmetries and liftings}} & \multicolumn{4}{|c|}{\text{expansions}}
\\
\midrule
\# &\; \lambda &\; d_\lambda &\; \text{rows} &\; \text{cols} &\; \text{sym} &\; \text{lift}  &\; \text{rows} &\; \text{cols} &\; \text{rank} &\; \text{null} &\; \text{new}
\\
\midrule
   1  &  10               &     1  &     9  &     8 &     8  &     8  &     8  &     1  &     0 &     8  &     \cdot  \\
   2  &  9 1            &     9  &    81  &    72 &    64  &    68  &    72  &     9  &     4 &    68  &     \cdot  \\
   3  &  8 2            &    35  &   315  &   280 &   222  &   264  &   280  &    35  &    15 &   265  &     1  \\
   4  &  8 1^2          &    36  &   324  &   288 &   218  &   269  &   288  &    36  &    19 &   269  &     \cdot  \\
   5  &  7 3            &    75  &   675  &   600 &   442  &   555  &   600  &    75  &    40 &   560  &     5  \\
   6  &  7 2 1        &   160  &  1440  &  1280 &   896  &  1197  &  1280  &   160  &    80 &  1200  &     3  \\
   7  &  7 1^3          &    84  &   756  &   672 &   446  &   632  &   672  &    84  &    40 &   632  &     \cdot  \\
   8  &  6 4            &    90  &   810  &   720 &   524  &   677  &   720  &    90  &    40 &   680  &     3  \\
   9  &  6 3 1        &   315  &  2835  &  2520 &  1722  &  2352  &  2520  &   315  &   160 &  2360  &     8  \\
  10  &  6 2^2          &   225  &  2025  &  1800 &  1200  &  1693  &  1800  &   225  &   105 &  1695  &     2  \\
  11  &  6 2 1^2      &   350  &  3150  &  2800 &  1798  &  2619  &  2800  &   350  &   180 &  2620  &     1  \\
  12  &  6 1^4          &   126  &  1134  &  1008 &   612  &   948  &  1008  &   126  &    60 &   948  &     \cdot  \\
  13  &  5^2              &    42  &   378  &   336 &   236  &   309  &   336  &    42  &    26 &   310  &     1  \\
  14  &  5 4 1        &   288  &  2592  &  2304 &  1568  &  2155  &  2304  &   288  &   144 &  2160  &     5  \\
  15  &  5 3 2        &   450  &  4050  &  3600 &  2390  &  3365  &  3600  &   450  &   230 &  3370  &     5  \\
  16  &  5 3 1^2      &   567  &  5103  &  4536 &  2960  &  4258  &  4536  &   567  &   276 &  4260  &     2  \\
  17  &  5 2^2 1        &   525  &  4725  &  4200 &  2698  &  3935  &  4200  &   525  &   265 &  3935  &     \cdot  \\
  18  &  5 2 1^3      &   448  &  4032  &  3584 &  2240  &  3360  &  3584  &   448  &   224 &  3360  &     \cdot  \\
  19  &  5 1^5          &   126  &  1134  &  1008 &   604  &   942  &  1008  &   126  &    66 &   942  &     \cdot  \\
  20  &  4^2 2            &   252  &  2268  &  2016 &  1358  &  1899  &  2016  &   252  &   116 &  1900  &     1  \\
  21  &  4^2 1^2          &   300  &  2700  &  2400 &  1562  &  2239  &  2400  &   300  &   160 &  2240  &     1  \\
  22  &  4 3^2          &   210  &  1890  &  1680 &  1124  &  1568  &  1680  &   210  &   110 &  1570  &     2  \\
  23  &  4 3 2 1    &   768  &  6912  &  6144 &  4032  &  5760  &  6144  &   768  &   384 &  5760  &     \cdot  \\
  24  &  4 3 1^3      &   525  &  4725  &  4200 &  2702  &  3940  &  4200  &   525  &   260 &  3940  &     \cdot  \\
  25  &  4 2^3          &   300  &  2700  &  2400 &  1582  &  2260  &  2400  &   300  &   140 &  2260  &     \cdot  \\
  26  &  4 2^2 1^2      &   567  &  5103  &  4536 &  2944  &  4245  &  4536  &   567  &   291 &  4245  &     \cdot  \\
  27  &  4 2 1^4      &   350  &  3150  &  2800 &  1810  &  2630  &  2800  &   350  &   170 &  2630  &     \cdot  \\
  28  &  4 1^6          &    84  &   756  &   672 &   442  &   627  &   672  &    84  &    44 &   628  &     1  \\
  29  &  3^3 1            &   210  &  1890  &  1680 &  1132  &  1580  &  1680  &   210  &   100 &  1580  &     \cdot  \\
  30  &  3^2 2^2          &   252  &  2268  &  2016 &  1338  &  1880  &  2016  &   252  &   136 &  1880  &     \cdot  \\
  31  &  3^2 2 1^2      &   450  &  4050  &  3600 &  2402  &  3380  &  3600  &   450  &   220 &  3380  &     \cdot  \\
  32  &  3^2 1^4          &   225  &  2025  &  1800 &  1184  &  1680  &  1800  &   225  &   120 &  1680  &     \cdot  \\
  33  &  3 2^3 1        &   288  &  2592  &  2304 &  1568  &  2160  &  2304  &   288  &   144 &  2160  &     \cdot  \\
  34  &  3 2^2 1^3      &   315  &  2835  &  2520 &  1726  &  2365  &  2520  &   315  &   155 &  2365  &     \cdot  \\
  35  &  3 2 1^5      &   160  &  1440  &  1280 &   896  &  1200  &  1280  &   160  &    80 &  1200  &     \cdot  \\
  36  &  3 1^7          &    36  &   324  &   288 &   222  &   271  &   288  &    36  &    16 &   272  &     1  \\
  37  &  2^5              &    42  &   378  &   336 &   244  &   320  &   336  &    42  &    16 &   320  &     \cdot  \\
  38  &  2^4 1^2          &    90  &   810  &   720 &   516  &   670  &   720  &    90  &    50 &   670  &     \cdot  \\
  39  &  2^3 1^4          &    75  &   675  &   600 &   446  &   565  &   600  &    75  &    35 &   565  &     \cdot  \\
  40  &  2^2 1^6          &    35  &   315  &   280 &   218  &   260  &   280  &    35  &    20 &   260  &     \cdot  \\
  41  &  2 1^8          &     9  &    81  &    72 &    64  &    68  &    72  &     9  &     4 &    68  &     \cdot  \\
  42  &  10             &     1  &     9  &     8 &     8  &     8  &     8  &     1  &     0 &     8  &     \cdot  \\
\midrule
\end{array}
\]
\medskip
\caption{$S_{10}$-module multiplicities for the anti-tetrad in degree 10}
\label{antitetrad10table}
\end{table}

In the rest of this section we present some nonlinear special identities in degree 10 for the anti-tetrad.
These identities correspond to the partitions $\lambda = s t$ ($s+t=10$) for which column ``new'' in
Table \ref{antitetrad10table} is nonzero, namely $\lambda = 82, 73, 64, 5^2$.
The corresponding Young diagrams have only two rows of lengths $s$ and $t$.
For these $\lambda$, the structure theory of $\field S_{10}$ allows us to assume that the identity is
a symmetric function of two disjoint sets of variables, so we may reduce the size
of the computations by working with nonlinear monomials which have only two variables $a$ and $b$
occurring respectively $s$ and $t$ times.

\begin{theorem} \label{antitetrad10special}
The multihomogeneous identities in Figure \ref{antispecial10} are satisfied by the anti-tetrad in every associative
algebra but are not consequences of the defining identities for AJQS.
\end{theorem}

\begin{remark}
For the anti-tetrad, we obtained better results using the LLL algorithm with a higher value
of the reduction parameter.
We illustrate with $\lambda = 64$; the underlying variables are $a^6 b^4$ which produce 210 associative monomials.
The eight association types produce respectively 82, 82, 82, 82, 32, 32, 14, 14 nonlinear quaternary monomials
for a total of 420.
The $210 \times 420$ expansion matrix has rank 99 and nullity 321.
Computing the HNF with an initial application of LLL produces a lattice basis of the nullspace with
size $\approx 421$ using equation \eqref{basissize}; the coefficients belong to $\{ 0, \pm 1, \pm 2, -3 \}$.
Applying LLL with parameter $3/4$ reduces the basis size to $\approx 389$ with coefficients $\{ 0, \pm 1, \pm 2 \}$.
Applying LLL again with parameter $99/100$ reduces the basis size to $\approx 337$ with coefficients $\{ 0, \pm 1, - 2 \}$.
We obtain three special nonlinear identities for this partition.
\end{remark}

\begin{figure}
\begin{align*}
  &
  [[[a, a, b, a], a, b, a], a, a, a]
  - [a, [[a, a, a, b], a, b, a], a, a]
  - [a, [a, [a, a, b, a], a, a], b, a]
  \\
  &
  {}
  + [a, [a, [a, a, b, a], b, a], a, a]
  \equiv 0,
  \\[4pt]
  &
  [[[a, a, b, a], a, b, a], b, a, a]
  + [[a, [a, b, a, b], a, a], b, a, a]
  - [a, [[a, a, b, a], b, a, a], b, a]
  \\
  &
  {}
  - [a, [[a, b, a, b], a, a, b], a, a]
  - [a, [[a, b, a, b], a, b, a], a, a]
  - [[a, a, b, a], [a, b, a, b], a, a]
  \equiv 0,
  \\[4pt]
  &
  [[[a, a, b, a], b, b, a], a, a, a]
  - [a, [[a, a, a, b], a, b, b], a, a]
  + [a, [[a, a, b, a], a, b, b], a, a]
  \\
  &
  {}
  - [a, [[a, a, b, a], b, b, a], a, a]
  + [a, [[a, a, b, b], a, a, b], a, a]
  - [a, [[a, a, b, b], a, b, a], a, a]
  \equiv 0,
  \\[4pt]
  &
  [[[a, a, b, a], a, b, a], a, a, b]
  - [a, [[a, a, a, b], a, b, a], a, b]
  + [a, [[a, a, b, a], a, a, b], a, b]
  \\
  &
  {}
  - [b, [[a, a, b, a], a, b, a], a, a]
  - [a, [a, [a, a, b, a], a, b], a, b]
  + [a, [a, [a, a, b, a], b, a], a, b]
  \equiv 0,
  \\[4pt]
  &
  [[a, [a, a, b, a], a, a], b, b, a]
  - [[a, [a, a, b, a], a, b], a, a, b]
  + [[a, [a, a, b, a], a, b], b, a, a]
  \\
  &
  {}
  + [[a, [a, a, b, b], a, a], b, a, a]
  - [a, [[a, a, b, a], a, b, b], a, a]
  - [a, [[a, a, b, b], a, b, a], a, a]
  \\
  &
  {}
  + [[a, a, a, b], [a, a, b, a], a, b]
  - [[a, a, a, b], b, [a, a, b, a], a]
  \equiv 0,
  \\[4pt]
  &
  [[[a, a, a, b], a, a, a], b, b, a]
  - [[[a, a, a, b], a, b, b], a, a, a]
  + [[a, [a, a, a, b], a, b], b, a, a]
  \\
  &
  {}
  - [[a, [a, a, b, a], a, a], b, b, a]
  + [[a, [a, a, b, b], a, a], a, a, b]
  - [[a, [a, a, b, b], b, a], a, a, a]
  \\
  &
  {}
  + [a, [[a, a, b, a], b, b, a], a, a]
  - [a, [[a, a, b, b], a, a, b], a, a]
  + [b, [a, [a, a, a, b], a, b], a, a]
  \\
  &
  {}
  - [[a, a, a, b], [a, a, b, b], a, a]
  \equiv 0,
  \\[4pt]
  &
  [[b, [a, a, a, b], a, b], a, a, b]
  + [[b, [a, a, b, a], a, b], a, a, b]
  - [b, [[a, a, a, b], b, a, a], a, b]
  \\
  &
  {}
  - [b, [[a, a, b, a], a, a, b], a, b]
  \equiv 0,
  \\[4pt]
  &
  [[[a, a, b, a], b, b, a], b, a, a]
  - [[a, [a, b, a, b], a, a], b, b, a]
  - [[a, [a, b, a, b], b, a], a, b, a]
  \\
  &
  {}
  - [a, [[a, a, b, a], b, a, b], b, a]
  + [a, [[a, b, a, b], a, a, b], b, a]
  \equiv 0,
  \\[4pt]
  &
  [[[a, a, b, a], b, b, a], b, a, a]
  - [a, [[a, a, b, a], b, b, a], b, a]
  - [a, [[a, b, a, b], a, b, b], a, a]
  \\
  &
  {}
  + [a, [[b, a, b, b], a, a, b], a, a]
  - [a, [[b, a, b, b], a, b, a], a, a]
  \equiv 0,
  \\[4pt]
  &
  [[[a, a, b, a], b, b, a], b, a, b]
  - [a, [[a, b, a, b], a, b, b], a, b]
  + [a, [[b, a, b, b], a, a, b], a, b]
  \\
  &
  {}
  - [a, [[b, a, b, b], a, b, a], a, b]
  - [b, [[a, a, b, a], b, b, a], b, a]
  \equiv 0.
  \\[-24pt]
\end{align*}
\caption{Some special identities for the anti-tetrad in degree 10}
\label{antispecial10}
\end{figure}


\section{Four families of finite dimensional anti-Jordan quadruple systems;\\
examples of universal associative envelopes}
\label{sectionUantitetrad}

We define four families of finite dimensional AJQS, denoted $A^-_n$, $B^-_n$, $C^-_{pqr}$, $D^-_{pq}$;
these are subspaces of matrix algebras which are closed under the anti-tetrad.
  \begin{itemize}
  \item
  $A^-_n$ consists of all $n \times n$ matrices; the dimension is $n^2$.
  \item
  $B^-_n$ consists of all $n \times n$ skew-symmetric matrices; the dimension is $\tfrac12 n(n{-}1)$.
  \end{itemize}
For the remaining systems, we choose integers $p, q, r \ge 1$ and consider block matrices of size $p+q+r$
as in equation \eqref{pqrmatrix}.
  \begin{itemize}
  \item
  $C^-_{pqr}$ ($p \ge q, r$) consists of all matrices \eqref{pqrmatrix}; the dimension is $pq+qr+rp$.
  \item
  $D^-_{pq}$ consists of all matrices \eqref{pqrmatrix} with $q = r$ where $M_{pq} = -M_{qp}^t$ and $M_{qq}$ is skew-symmetric;
  the dimension is $pq + \tfrac12 q(q{-}1)$.
  \end{itemize}
The smallest non-trivial systems in each family are $A^-_2$, $B^-_3$, $C^-_{111}$, $D^-_{21}$ with dimensions
4, 3, 3, 2 respectively.
We consider these systems by increasing dimension.

\begin{proposition}
The universal associative envelope $U( D^-_{21} )$ is $\mathbb{Z}$-graded by degree and infinite dimensional;
in degrees $n \ge 6$ it is linearly isomorphic to the (commutative) polynomial algebra $\mathbb{F}[a,b]$.
\end{proposition}

\begin{proof}
This is the basis for $D^-_{21}$ which defines the natural representation:
  \[
  a = \begin{bmatrix} 0 & 0 & 0 & -1 \\ 0 & 0 & 0 & 0 \\ 1 & 0 & 0 & 0 \\ 0 & 0 & 0 & 0 \end{bmatrix},
  \qquad
  b = \begin{bmatrix} 0 & 0 & 0 & 0 \\ 0 & 0 & 0 & -1 \\ 0 & 1 & 0 & 0 \\ 0 & 0 & 0 & 0 \end{bmatrix}.
  \]
Every quadrilinear product in these generators is 0, so the original 6 ideal generators have no terms of degree 1;
it follows that $U( D^-_{21} )$ is graded by degree:
  \[
  aba^2 - a^2ba, \quad
  ba^3 - a^3b, \quad
  baba - abab, \quad
  b^2a^2 - a^2b^2, \quad
  b^2ab - bab^2, \quad
  b^3a - ab^3.
  \]
These generators produce 12 distinct non-trivial compositions:
  \begin{align*}
  &
  a^3ba - a^4b, \quad
  a^2b^2a - a^3b^2, \quad
  ba^2ba - a^2bab, \quad
  ba^2b^2 - a^2b^3,
  \\
  &
  bab^2a - abab^2, \quad
  bab^3 - ab^4, \quad
  a^2bab^2 - a^3b^3, \quad
  ba^2bab - abab^2a,
  \\
  &
  a^4bab - a^5b^2, \quad
  a^3bab^2 - a^4b^3, \quad
  abab^4 - a^2b^5, \quad
  ba^2bab^2 - a^3b^4.
  \end{align*}
After self-reduction, the 18 generators become a Gr\"obner basis of 13 elements:
  \begin{align*}
  &
  aba^2 {-} a^2ba, \quad
  ba^3 {-} a^3b, \quad
  baba {-} abab, \quad
  b^2a^2 {-} a^2b^2, \quad
  b^2ab {-} bab^2, \quad
  b^3a {-} ab^3, \quad
  a^3ba {-} a^4b,
  \\
  &
  a^2b^2a {-} a^3b^2, \quad
  ba^2ba {-} a^2bab, \quad
  ba^2b^2 {-} a^2b^3, \quad
  bab^2a {-} abab^2, \quad
  bab^3 {-} ab^4, \quad
  a^2bab^2 {-} a^3b^3.
  \end{align*}
There is an infinite set of monomials in $F\langle a,b \rangle$ which are not divisible by the leading monomial
of any element of the Gr\"obner basis, and $U( D^-_{21} )$ is infinite dimensional.
Here are the basis monomials for $U( D^-_{21} )$ in degrees $\le 5$:
  \begin{align*}
  &
  1, \quad
  a, \quad
  b, \quad
  a^2, \quad
  ab, \quad
  ba, \quad
  b^2, \quad
  a^3, \quad
  a^2b, \quad
  aba, \quad
  ab^2, \quad
  ba^2, \quad
  bab, \quad
  b^2a, \quad
  b^3,
  \\
  &
  a^4, \quad
  a^3b, \quad
  a^2ba, \quad
  a^2b^2, \quad
  abab, \quad
  ab^2a, \quad
  ab^3, \quad
  ba^2b, \quad
  bab^2, \quad
  b^4,
  \\
  &
  a^5, \quad
  a^4b, \quad
  a^3b^2, \quad
  a^2bab, \quad
  a^2b^3, \quad
  abab^2, \quad
  ab^4, \quad
  b^5.
  \end{align*}
In degree $n \ge 6$ a basis consists of the monomials $a^{n-i} b^i$ for $i = 0, \dots, n$.
To see this, note that every element of the Gr\"obner basis has the form $m_1 - m_2$,
and replacing $m_1$ by $m_2$ moves $a$ to the left and $b$ to the right;
now use induction on $n$.
Thus in degrees $n \ge 6$, $U(D_{21}^-)$ is isomorphic to the (commutative) polynomial algebra $\field[a,b]$,
but there is some noncommutativity in degrees $n \le 5$.
\end{proof}

\begin{remark}
We did similar computations for $D^-_{31}$.
The original set of 36 homogeneous generators produced 278 distinct non-trivial compositions;
after self-reduction the 314 generators became a Gr\"obner basis of 94 elements.
It follows that $U( D^-_{31} )$ is $\mathbb{Z}$-graded and infinite dimensional,
and that for degree $n \ge 6$, the dimension of the homogeneous subspace is $\binom{n+2}{2}$.
\end{remark}

\begin{conjecture}
For all $p \ge 2$, the universal associative envelope $U( D^-_{p1} )$ is $\mathbb{Z}$-graded and infinite dimensional,
and for degree $n \ge 6$, the dimension of the homogeneous component is $\binom{n+p-1}{p-1}$, equal to that of
the (commutative) polynomial algebra in $p$ variables.
\end{conjecture}

\begin{openproblem}
Study the representation theory of the anti-Jordan triple systems $D^-_{p1}$
and their universal associative envelopes $U( D^-_{p1} )$.
Comparing our results with \cite{Elgendy1} suggests that
$U( D^-_{p1} )$ is likely to be a (generalized) down-up algebra.
\end{openproblem}

\begin{remark}
We did similar computations for $D^-_{12}$.
The original set of 36 non-homogeneous generators produced 341 distinct non-trivial compositions;
after self-reduction the 377 generators became 19 elements.
This set produced one more non-trivial composition, resulting in a Gr\"obner basis of 20 elements:
  \begin{align*}
  &
  c^2, \quad
  a^3, \quad
  a^2b, \quad
  aba, \quad
  ab^2, \quad
  aca, \quad
  ba^2, \quad
  bab, \quad
  b^2a, \quad
  b^3,
  \\
  &
  bca + acb, \quad
  bcb, \quad
  cac, \quad
  cbc, \quad
  acba + a^2cb - a, \quad
  bacb + acb^2 - b,
  \\
  &
  cbac - cabc - c, \quad
  a^2cb^2 - ab, \quad
  acb^2c - bc, \quad
  ca^2cb - ca.
  \end{align*}
From this it follows that $U( D^-_{12} )$ has dimension 26.
\end{remark}

\begin{openproblem}
Study the representation theory of the anti-Jordan triple systems $D^-_{pq}$ for $q \ge 2$.
In particular, do we always have $U( D^-_{pq} ) \cong \field \oplus M_{p+2q}(\field)$?
\end{openproblem}

\begin{proposition}
We have $U( C^-_{111} ) \cong \field \oplus 2\,M_3(\field)$, and hence up to isomorphism,
$C^-_{111}$ has only three finite dimensional irreducible representations: 1-dimensional trivial,
3-dimensional natural, and another 3-dimensional.
\end{proposition}

\begin{proof}
Our standard basis for $C^-_{111}$ is the same as that for $C_{111}$; see the proof of Proposition \ref{propUC111}.
Only three quadrilinear products are nonzero:
  \[
  [ a, c, b, a ] = a, \qquad
  [ b, c, a, b ] = -b, \qquad
  [ c, b, a, c ] = c.
  \]
The original set $G$ of generators for the ideal $I(G)$ contains 36 elements and is already self-reduced.
We obtain 333 distinct nontrivial compositions of these generators;
the resulting set of 369 generators collapses to a self-reduced set of only 13 elements, which is a Gr\"obner basis:
  \begin{align*}
  &
  a^2, \quad
  b^2, \quad
  c^2, \quad
  aba, \quad
  aca, \quad
  bab, \quad
  bcb, \quad
  cac, \quad
  cbc,
  \\
  &
  acba - abca - a, \quad
  bcab - bacb + b, \quad
  cbac - cabc - c, \quad
  cabca + ca.
  \end{align*}
The leading monomials of these elements are the same as those of the Gr\"obner basis in the proof of Proposition \ref{propUC111}.
Hence $U( C^-_{111} )$ has dimension 19, and a basis consists of the same monomials as for $U( C_{111} )$ in the proof of Proposition \ref{propUC111}.
We find that $U( C^-_{111} )$ is semisimple, and its center has dimension 3.
The cosets of these elements form a central basis of orthogonal primitive idempotents:
  \[
  1 + abc - acb - bac + bca + cab - cba, \qquad
  - abc  - bca - cab, \qquad
  acb + bac + cba.
  \]
These elements generate simple two-sided ideals of dimensions 1, 9, 9 respectively.
We omit the isomorphisms of each simple ideal with a full matrix algebra.
\end{proof}

\begin{openproblem}
Study the representation theory of the anti-Jordan triple systems $C^-_{pqr}$.
In particular, do we always have $U( C^-_{pqr} ) \cong \field \oplus 2\,M_{p+q+r}(\field)$?
\end{openproblem}

\begin{proposition}
We have $U( B^-_3 ) \otimes_\field \mathbb{K} \cong \mathbb{K} \oplus 3\,M_3(\mathbb{K})$ where $\mathbb{K} = \field(\beta)$,
$\beta=1+\sqrt{-3}$.
Hence up to isomorphism,
$B^-_3$ has four finite dimensional irreducible representations over $\mathbb{K}$:
1-dimensional trivial, 3-dimensional natural, and two other 3-dimensional.
\end{proposition}

\begin{proof}
Our standard basis for $B^-_3$ consists of these three $3 \times 3$ matrices defining the natural representation:
  \[
  a = \begin{bmatrix} 0 & -1 & 0 \\ 1 & 0 & 0 \\ 0 & 0 & 0 \end{bmatrix}, \qquad
  b = \begin{bmatrix} 0 & 0 & -1 \\ 0 & 0 & 0 \\ 1 & 0 & 0 \end{bmatrix}, \qquad
  c = \begin{bmatrix} 0 & 0 & 0 \\ 0 & 0 & -1 \\ 0 & 1 & 0 \end{bmatrix}.
  \]
The original set $G$ of generators for the ideal $I(G)$ contains 36 elements and is already self-reduced.
We obtain 385 distinct nontrivial compositions of these generators;
the resulting set of 421 generators collapses to a self-reduced set of 30 elements, which is a Gr\"obner basis:
  \begin{align*}
  &
  aba, \quad
  aca, \quad
  ac^2 + ab^2 - a^3, \quad
  bab, \quad
  b^2a + ab^2 - a^3, \quad
  b^3 - ba^2 - a^2b, \quad
  bca + acb,
  \\
  &
  bcb, \quad
  bc^2 - a^2b, \quad
  ca^2 - b^2c, \quad
  cab + bac, \quad
  cac, \quad
  cba + abc, \quad
  cb^2 - a^2c, \quad
  cbc,
  \\
  &
  c^2a - ab^2, \quad
  c^2b - ba^2, \quad
  c^3 - b^2c - a^2c, \quad
  a^2cb - a^2bc - a, \quad
  ab^2c, \quad
  ba^3 - a^3b - c,
  \\
  &
  ba^2c, \quad
  bacb + a^3c - b, \quad
  a^5 + cb - bc, \quad
  a^4b + ac, \quad
  a^4c - ab, \quad
  a^3b^2 + cb,
  \\
  &
  a^3bc +\tfrac12 c^2 -\tfrac12 b^2 +\tfrac12 a^2, \quad
  ba^2b^2 - ca, \quad
  ba^2bc + ba.
  \end{align*}
Only 28 monomials in $F\langle a, b, c \rangle$ do not have a leading monomial of a Gr\"obner basis element as a subword:
  \begin{align*}
  &
  1, \quad
  a, \quad
  b, \quad
  c, \quad
  a^2, \quad
  ab, \quad
  ac, \quad
  ba, \quad
  b^2, \quad
  bc, \quad
  ca, \quad
  cb, \quad
  c^2, \quad
  a^3, \quad
  a^2b, \quad
  a^2c,
  \\
  &
  ab^2, \quad
  abc, \quad
  acb, \quad
  ba^2, \quad
  bac, \quad
  b^2c, \quad
  a^4, \quad
  a^3b, \quad
  a^3c, \quad
  a^2b^2, \quad
  a^2bc, \quad
  ba^2b.
  \end{align*}
Hence $U( B^-_3 )$ has dimension 28; as before, we find that $U( B^-_3 )$ is semisimple.
The center has dimension 4; finding a basis of orthogonal primitive idempotents requires a quadratic extension of $\field$.
The required basis is
  \begin{align*}
  &\!\!\!\!
  1 +  abc -  acb -  bac,
  \\
  -\tfrac16
  &\big( a^2 + b^2 + c^2 + 2  abc - 2  acb - 2  bac - 2  a^4 - 2  ba^2b \big),
  \\
  \tfrac1{12}
  &\big( \beta ( a^2 + b^2 + c^2 ) -4  abc +4  acb +4  bac -2\bar\beta ( a^4 + ba^2b ) \big),
  \\
  \tfrac1{12}
  &\big( \bar\beta ( a^2 + b^2 + c^2 ) -4  abc +4  acb +4  bac -2\beta ( a^4 + ba^2b ) \big),
  \end{align*}
where $\beta = 1+\sqrt{-3}$.
These four elements generate simple two-sided ideals of dimensions 1, 4, 4, 4 respectively.
We omit the calculation of the isomorphisms of each simple ideal with a full matrix algebra.
\end{proof}

\begin{openproblem}
Study the representation theory of the anti-Jordan triple systems $B^-_n$.
In particular, over a quadratic extension $\mathbb{K}$ of $\field$ do we always have
$U( B^-_n ) \otimes_\field \mathbb{K} \cong \mathbb{K} \oplus 3\,M_n(\mathbb{K})$?
\end{openproblem}

\begin{proposition}
We have $U( A^-_2 ) \otimes_\field \mathbb{K} \cong \mathbb{K} \oplus 6\,M_2(\mathbb{K})$ where $\mathbb{K} = F(\beta)$,
$\beta = 1+\sqrt{-3}$.
Hence up to isomorphism,
$A^-_2$ has seven finite dimensional irreducible representations over $\mathbb{K}$:
1-dimensional trivial, 2-dimensional natural, and five other 2-dimensional.
\end{proposition}

\begin{proof}
Our standard basis for $A^-_2$ consists of the four $2 \times 2$ matrix units as in the proof of Proposition \ref{propUA2}.
The original set $G$ of generators for the ideal $I(G)$ contains 120 elements and is already self-reduced.
We obtain 2821 distinct nontrivial compositions of these generators;
the resulting set of 2941 generators collapses after three iterations to a self-reduced set of only 28 elements, which is a Gr\"obner basis:
  \begin{align*}
  &
  ad, \quad
  b^2, \quad
  bd - ab, \quad
  c^2, \quad
  cd - ac, \quad
  da, \quad
  db - ba, \quad
  dc - ca,
  \\
  &
  d^2 - cb - bc + a^2, \quad
  aba, \quad
  aca, \quad
  acb + abc - a^3, \quad
  bab, \quad
  bca - abc,
  \\
  &
  bcb - ba^2 - a^2b, \quad
  cac, \quad
  cba + abc - a^3, \quad
  cbc - ca^2 - a^2c, \quad
  a^2bc - \tfrac12 a^4 - \tfrac12 a,
  \\
  &
  ba^3 - a^3b + b, \quad
  ba^2b, \quad
  ca^3 - a^3c - c, \quad
  ca^2b - ba^2c - d, \quad
  ca^2c,
  \\
  &
  cabc - a^3c - c, \quad
  a^4b - ab, \quad
  a^4c + ac, \quad
  a^6 -2 abc + a^3.
  \end{align*}
Only 25 monomials in $F\langle a, b, c, d \rangle$ do not have a leading monomial of a Gr\"obner basis element as a subword:
  \begin{align*}
  &
  1, \quad
  a, \quad
  b, \quad
  c, \quad
  d, \quad
  a^2, \quad
  ab, \quad
  ac, \quad
  ba, \quad
  bc, \quad
  ca, \quad
  cb, \quad
  a^3, \quad
  a^2b,
  \\
  &
  a^2c, \quad
  abc, \quad
  ba^2, \quad
  bac, \quad
  ca^2, \quad
  cab, \quad
  a^4, \quad
  a^3b, \quad
  a^3c, \quad
  ba^2c, \quad
  a^5.
  \end{align*}
Hence $U( A^-_2 )$ has dimension 25; as before, we find that $U( A^-_2 )$ is semisimple.
The center has dimension 7; splitting it requires a quadratic extension of $\field$.
We obtain this central basis of orthogonal primitive idempotents:
  \begin{align*}
  &\!\!\!\!
  1 + a^3 - 2 abc + bac - cab,
  \\
  \tfrac16 &\big( a + 2 d + 2 cb + 2 abc + 2 cab + a^4 + 2 ba^2c + 2 a^5 \big),
  \\
  -\tfrac16 &\big( a -2 bc + 2 a^3 - 2 abc + 2 bac - a^4 - 2 ba^2c + 2 a^5 \big),
  \\
  \tfrac1{12} &\big(
        \beta  a
  -2\bar\beta  bc
           -4  a^3
           +4  abc
           -4  bac
       -\beta  a^4
      -2\beta  ba^2c
  +2\bar\beta  a^5
  \big),
  \\
  \tfrac1{12} &\big(
    \bar\beta  a
      -2\beta  bc
           -4  a^3
           +4  abc
           -4  bac
   -\bar\beta  a^4
  -2\bar\beta  ba^2c
      +2\beta  a^5
  \big),
  \\
  -\tfrac1{12} &\big(
        \beta  a
      +2\beta  d
  +2\bar\beta  cb
           -4  abc
           -4  cab
       +\beta  a^4
      +2\beta  ba^2c
  +2\bar\beta  a^5
  \big),
  \\
  -\tfrac1{12} &\big(
    \bar\beta  a
  +2\bar\beta  d
      +2\beta  cb
           -4  abc
           -4  cab
   +\bar\beta  a^4
  +2\bar\beta  ba^2c
      +2\beta  a^5
  \big),
  \end{align*}
where $\beta = 1 + \sqrt{-3}$.
These seven elements generate simple two-sided ideals of dimensions 1, 4, 4, 4, 4, 4, 4 respectively.
We omit the calculation of the isomorphisms of each simple ideal with a full matrix algebra.
\end{proof}

\begin{openproblem}
Study the representation theory of the anti-Jordan triple systems $A^-_n$.
In particular, over a quadratic extension $\mathbb{K}$ of $\field$ do we always have
$U( A^-_n ) \otimes_\field \mathbb{K} \cong \mathbb{K} \oplus 6\,M_n(\mathbb{K})$?
\end{openproblem}


\end{document}